\documentclass[noinfoline]{imsart}

\usepackage{array}

\usepackage{amsmath}
\usepackage{amssymb}
\usepackage{amsthm}
\usepackage{amscd}
\usepackage{amsfonts}
\usepackage[authoryear]{natbib}
\usepackage{graphicx,caption, subcaption}%
\usepackage{fancyhdr, geometry}
\usepackage[ruled, % to add horizontal lines.
lined, %boxed,
commentsnumbered]{algorithm2e}
\usepackage{tikz}
\usetikzlibrary{positioning}
\usepackage{url, hyperref, color, verbatim}%, ulem}
\usepackage[normalem]{ulem}

\newcommand{\conv}{conv}
\newcommand{\mP}{\mathcal P}
\newcommand{\G}{\mathcal G}
\DeclareMathOperator{\dgen}{dgen}

\newtheorem{thm}{Theorem}
\newtheorem{defn}[thm]{Definition}
\newtheorem{example}[thm]{Example}
\newtheorem{proposition}[thm]{Proposition}
\newtheorem{condition}[thm]{Condition}
\newtheorem{lm}[thm]{Lemma}
\theoremstyle{definition}
\newtheorem{rmk}[thm]{Remark}

\begin{document}

\title{
	Statistical models for cores decomposition of an undirected random graph}

\author{  	Vishesh Karwa\thanks{Carnegie Mellon University and Harvard University, {\tt vishesh@cmu.edu}},
	Michael J. Pelsmajer\thanks{Illinois Institute of Technology, {\tt pelsmajer@iit.edu}},
	Sonja Petrovi\'c\thanks{Illinois Institute of Technology, {\tt sonja.petrovic@iit.edu}} ,
	Despina Stasi\thanks{Illinois Institute of Technology, {\tt stasdes@iit.edu}},
	Dane Wilburne\thanks{Illinois Institute of Technology, {\tt dwilburn@hawk.iit.edu}}
}\thanks{Authors are listed in alphabetical order; contributions are equal.}

\begin{abstract}
	The $k$-core decomposition is a widely studied
summary statistic that describes a graph's
global connectivity  
structure. In this paper, we move beyond using $k$-core decomposition as a 
  tool to summarize a graph and propose using $k$-core decomposition as a tool to model random graphs. 	
We propose using the shell distribution vector, a way of summarizing the decomposition, as a sufficient statistic for a family of exponential random graph models.
We study the properties and behavior of the model family, implement a Markov chain Monte Carlo algorithm for simulating graphs from the model, implement a direct sampler from the set of graphs with a given shell distribution, and explore the sampling distributions of some of the commonly used complementary statistics as good candidates for heuristic model fitting.
These algorithms  provide first fundamental steps necessary for solving the following problems:  parameter estimation in this ERGM,  extending the model to its Bayesian relative, and developing a rigorous methodology for testing goodness of fit of the model and model selection.
 The methods are applied to a synthetic network as well as the well-known Sampson monks dataset.
\end{abstract}
\maketitle

\section{Introduction}

Network analyses are often concerned---either directly or indirectly---with the degrees of the nodes in the network, a natural approach since counting the number of edges incident to a node gives a basic local measure of connectivity.
Several familiar statistical frameworks assign a probability distribution  to the set of  networks on a fixed number of nodes  based on their degree information,  e.g. \cite{holland1981exponential}, \cite{CDS11}, \cite{OW2012}, and \cite{rinaldo2013maximum}. However, despite the rich structure degree-based models offer compared to simpler  models such as Erd\"os-Renyi-Gilbert, they fail to capture
certain vital connectivity information about the network. In some applications, it matters not just to how many other nodes a particular node in the network is connected, but also to \emph{which} other nodes it is connected. For example, a node $v$ may seem important if it has high degree, but if all its neighbors are themselves unimportant due to having no additional connections (e.g., if they all have degree~1), then the ``influence'' or ``centrality'' of $v$ within the network is not actually all that impressive, after all.
  This distinction is especially crucial in applications concerning information dispersal as in \cite{PMAZM14}, the spread of infectious diseases or viruses as in \citet{kitsak2010identification}, or robustness to node failure. In the social network context, this importance can be interpreted as ``celebrity status'' of a node. Whereas degree-centric analyses are not well-suited to model such situations, the \emph{core decomposition} of a network graph can capture precisely this type of information.

Cores of a graph were introduced by \cite{Seidman83} to study tightly-knit groups in social networks. Since then, core decomposition 
has been used as a tool for numerous applications varying from understanding protein networks \citep{wuchty2005peeling}, visualization of large networks \citep{alvarez2006k}, and 	understanding the topology of the Internet graph \citep{CarmiPNAS07} to name a few.
In studies such as \cite{kitsak2010identification} and \cite{BK14}, the authors  identify spreader nodes and rank them  in terms of their spreading influence, using a graph's core decomposition. Methods for identifying spreaders using cores were extended to dynamic networks in \cite{MD10} and core decomposition in general was extended to weighted networks in \cite{EA13}. An important feature of a core decomposition is that it can be computed efficiently (see, e.g., \cite{LRATJZ13}), even for ``uncertain graphs'' which are graphs whose edges have some probability of existing--such graphs have applications in biological networks that model, for instance, protein interactions (see \cite{BGKV14}). Although core decomposition has become an important and widely used tool as a descriptive summary statistic of the network, it is a statistic for which there does not exist an associated statistical model.

The goal of this paper is to place the core decomposition of a network on a rigorous statistical foundation and present it as a tool for statistical modeling rather than descriptive analysis. We construct a natural model based on core decomposition by embedding the core structure of a graph in the family of exponential random graph models (ERGMs) and describe its theoretical properties. We restrict the support of the model to allow only networks with a fixed \emph{degeneracy} to have a positive probability. We show that this eliminates certain bad properties common to many ERGMs and expect that such support restrictions may help improve the properties of other ERGMs as well. We study three common inference tasks  as they apply to the support restricted ERGM: sampling, maximum likelihood estimation, and goodness-of-fit testing. More specifically, the contributions of this paper are as follows:
\begin{enumerate}
	\item In Section \ref{sec:prelim}, we summarize the core decomposition of a network in the form of a \emph{shell distribution}, and in Section \ref{sec:distributionModel} we introduce a \emph{support restricted} exponential random graph model with the shell distribution as a sufficient statistic. 
		\item In Section \ref{sec:mcmc}, we perform simulation studies to understand the behavior of the model by relying on an MCMC algorithm to sample from the model and to estimate the parameters of the model.
		\item In Section \ref{sec:samplingSequence}, we present an
		algorithm to sample from the space of graphs given a fixed shell distribution.
	\item We return to the theoretical properties of the model in Sections~\ref{sec:structural} and~\ref{sec:distributionMLE}, where we study the
	space of graphs with a fixed shell distribution and describe the \emph{marginal polytope} associated with the model and conditions for the existence of MLE, respectively.
\end{enumerate}

ERGMs provide a natural framework to model networks through their sufficient statistics; see \cite{robins2007introduction} for an introduction. \cite{F-review} provide a comprehensive review of various ways to model networks, including ERGMs.
 ERGMs are a special case of the venerable class of exponential families which are known to possess excellent statistical properties; see \cite{BR:86} for a theoretical treatment of exponential families and \cite{RFZ:mle09} in particular for discrete exponential family models, including ERGMs. ERGMs have been the workhorse of many applied studies, and the literature is too vast to be surveyed here; see~\cite{snijders2006new,saul2007exploring} and \cite{goodreau2009birds} for examples of studies that use ERGMs for network modeling.

Our goal is to add to the toolbox of ERGMs the ability to model the core structure of a graph.
Doing so has two important consequences: First, it puts the core structure of a graph, summarized by its shell distribution, on a firm statistical footing. Second, it  allows us to understand what properties of a network  are captured by the shell distribution.
	 It is worth noting that any ERGM based on a core decomposition  \emph{cannot} be specialized to the Erd\"os-R\'enyi model, i.e.,  the Erd\"os-R\'enyi model is not a submodel of any ERGM based on the core decomposition. 
 In fact, the same is true for any ERGM with sufficient statistics based on the degree sequence of the network. As such, the shell distribution ERGM would occupy a unique space in the network literature. Models based on the core distribution go beyond the dyadic independence assumption inherent in the degree sequence based network models and are able to capture transitivity effects. These models differ from the ERGM-based subgraph counts, such as triangles and stars, which also go beyond the dyadic independence assumption. This is because the core structure of a network is a \emph{global} sufficient statistic in the following sense: To which core a node belongs depends in some way on the entire network; see Section \ref{sec:prelim} for the precise definition of a core and some examples. In contrast, subgraph counts measure local and coarse properties of the network.

We want to point out that for all the good properties of ERGMs, they are not without drawbacks. Recent empirical and theoretical work has brought to light some undesirable properties of some special classes of ERGMs; these properties are often termed as ``model degeneracy'' 
 (\cite{RFZ:mle09,schweinberger2011instability, chatterjee2013estimating, hunter2008goodness}) or ``inconsistency'' (\cite{shalizi2013consistency}). As noted in \cite{RFZ:mle09}, ``model degeneracy'' is an umbrella term used to denote many undesirable properties of ERGMs.  One specific drawback to note is that it may be difficult to sample efficiently~(\cite{ergmsHard15}), but that is an issue for ERGMs in general and outside the scope of this paper.
 We discuss these issues in Section \ref{sec:appendixDegenerate} and explain  how we fix them by placing support restrictions on  the class of models that we consider. Since the word degeneracy also refers to a graph-theoretic notion which is relevant to this work, we avoid the use of the term ``model degeneracy'' and instead use the term ``bad'' behavior of the model.

\section{Technical preliminaries: cores and shells}
\label{sec:prelim}
We restrict our analysis to the set of \emph{simple} graphs, representing networks without multiple edges and self-loops. For the remainder of this manuscript, let $\G_n$ denote the set of all simple graphs on $n$ nodes. We are interested in distributions over the set $\G_n$; thus $G$ will denote a random variable with state space $\G_n$,  and $G=g$ its realization. We will also consider families of subsets of $\mathcal G_n$ below.
\begin{defn}[\cite{Seidman83}]\rm
\label{defn:k-core}
 The \emph{k-core} of a graph $g$, denoted by  $H_k(g)$ or simply $H_k$ if the graph is clear from the context, is the maximal subgraph in which every vertex has degree at least $k$\footnote{This is the usual definition of the $k$-core and it appropriately describes the notion of node importance and robust degree. Seidman's original definition also requires the subgraph to be connected.}.
\end{defn}
As it is often useful to think of the $k$-core as the output of an algorithm for which the graph $g$ is the input, we also use the equivalent algorithmic definition: $H_k$ is the subgraph obtained by iteratively deleting vertices of degree less than $k$; see Algorithm~\ref{alg:GetShellSeq}.
For example, for the particular graph $G=g$ on the left of Figure~\ref{fig:core-example},
$H_0(g)$ is just the graph itself, $H_1$ is $g$ without the isolated vertex, the $2$-core $H_2$ is shown in the middle,
and 
$H_3$ and $H_4$ are the same graph, shown on the right.  For $k\ge5$, $H_k$  is the empty graph.

\begin{figure}[h]
\begin{tikzpicture}
  [scale=.2,auto=center,every node/.style={circle,fill=black},inner sep=1.5pt]
  \node (n1) at (-9,0)  {}; %L
  \node (n2) at (5.5,0)  {};
  \node (n3) at (4,3)  {};
  \node (n4) at (8.5,0)  {};
  \node (n5) at (10,3)  {};
  \node (n6) at (5.5,6)  {};
  \node (n7) at (0,0)  {}; %L
  \node (n8) at (-6,0)  {}; %L
  \node (n9) at (-3,2)  {}; %L
  \node (n10) at (0,6)  {}; %M
  \node (n11) at (-2.121312,8.12132)  {}; %M
  \node (n12) at (-3,6)  {};
  \node (n13) at (0,9)  {};
  \node (n14) at (-2,3.5)  {};
  \node (n15) at (8.5,6)  {};
  \node (n16) at (-1.5,-2.5)  {}; %L
  \node (n17) at (-4.5,-2.5)  {}; %L
  \foreach \from/\to/\weight in {
  n15/n3/1, n15/n4/1, n15/n5/1, n15/n6/1, n2/n6/1, n3/n6/1, n5/n6/1,
  n2/n5/1, n2/n3/1, n3/n4/1, n2/n4/1, n3/n5/1, n4/n5/1,
  n3/n10/1,n10/n11/1,n10/n12/1,n10/n13/1,
  n7/n3/1,
  n1/n8/1, n7/n8/1, n7/n17/1,  n8/n9/1,n7/n9/1,  n7/n16/1,  n8/n17/1, n16/n17/1,
  n16/n2/1 }
    \draw (\from) --(\to);
\end{tikzpicture}
\quad\quad\quad\quad\quad
\begin{tikzpicture}
  [scale=.2,auto=center,every node/.style={circle,fill=black},inner sep=1.5pt]
  \node (n2) at (5.5,0)  {};
  \node (n3) at (4,3)  {};
  \node (n4) at (8.5,0)  {};
  \node (n5) at (10,3)  {};
  \node (n6) at (5.5,6)  {};
  \node (n7) at (0,0)  {}; %L
  \node (n8) at (-6,0)  {}; %L
  \node (n9) at (-3,2)  {}; %L
  \node (n15) at (8.5,6)  {};
  \node (n16) at (-1.5,-2.5)  {}; %L
  \node (n17) at (-4.5,-2.5)  {}; %L
  \foreach \from/\to/\weight in {
  n15/n3/1, n15/n4/1, n15/n5/1, n15/n6/1, n2/n6/1, n3/n6/1, n5/n6/1,
  n2/n5/1, n2/n3/1, n3/n4/1, n2/n4/1, n3/n5/1, n4/n5/1,
  n7/n3/1,
  n7/n8/1, n7/n17/1,  n8/n9/1, n7/n9/1,  n7/n16/1,  n8/n17/1, n16/n17/1,
  n16/n2/1 }
    \draw (\from) --(\to);
\end{tikzpicture}
\quad\quad\quad\quad\quad
\begin{tikzpicture}
  [scale=.2,auto=center,every node/.style={circle,fill=black},inner sep=1.5pt]
      \path [use as bounding box,red] (4,-2.5) rectangle (9,8.5);
  \node (n2) at (5.5,0)  {};
  \node (n3) at (4,3)  {};
  \node (n4) at (8.5,0)  {};
  \node (n5) at (10,3)  {};
  \node (n6) at (5.5,6)  {};
  \node (n15) at (8.5,6)  {};
  \foreach \from/\to/\weight in {
  n15/n3/1, n15/n4/1, n15/n5/1, n15/n6/1, n2/n6/1, n3/n6/1, n5/n6/1,
  n2/n5/1, n2/n3/1, n3/n4/1, n2/n4/1, n3/n5/1, n4/n5/1 }
    \draw (\from) --(\to);
\end{tikzpicture}		\caption{A small graph $g$ (left), its $2$-core (center), and its 3- and 4-core (right).}
		\label{fig:core-example}
\end{figure}
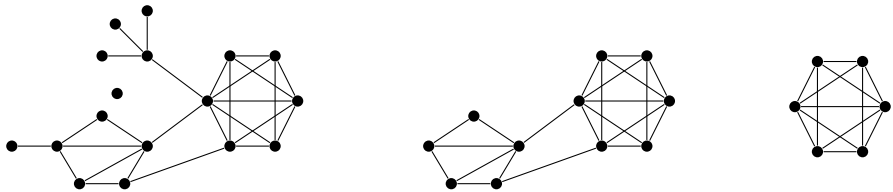

Each node is contained in several $k$-cores, for every $k$ from $0$ to whatever the largest $k$ is for that node.  Thus, the following node statistic captures all core information for a node.
\begin{defn}\label{defn:shell-index}\rm
A vertex $v$ in a graph $g$ has \emph{shell index i} if $v\in H_i(g)$ but $v\notin H_{i+1}(g)$.
Define $s_g:V\rightarrow\mathbb{N}$ as the function that maps vertices of $g$ to the non-negative integers according to their shell indices, so that if $v$ has shell index $i$ we may write $s_g(v)=i$.  If the graph $g$ is clear from the context, we drop the subscript and simply write $s(v)=i$.
\end{defn}
In other words, the shell index of a vertex $v$ indicates the highest core to which $v$ belongs.
For example, not all nodes in the $2$-core $H_2(g)$ in the middle of Figure~\ref{fig:core-example} have shell index $2$ in $g$: the six nodes on the right have shell index $4$.
The vertex set $V(g)$ of any network $g$ can be partitioned  according to the shell indices, since the shell index exists, is well-defined and is unique for all vertices.
There are two natural ways to record all of the shell index information about a network, and hence, record the information that captures its core structure.
First,
the \emph{shell sequence} $s(g)$ of an $n$-vertex graph $g$
 with vertices $v_1,\ldots,v_n$ is a vector of length $n$ whose $i^{th}$ entry is the shell index of vertex $v_i$.
  Second,
if the interest is in unlabeled graphs (i.e., exchangeable models for labeled graphs),  it is natural to summarize the sequence with a histogram as follows.
The \emph{shell distribution $n_S(g)$} of an $n$-vertex graph $g$ is a vector of length $n$ whose $j^{th}$ entry $n_j(g)$ is the number of vertices of $g$ that have shell index $j$, for $0\leq j\leq n-1$.  (The shell index of a vertex is bounded above by its degree, which is bounded above by $n-1$.)  Note that  $\sum_{j=0}^{n-1} n_{j}(g)=n$.
  In symbols,
 \[
 	n_S(g):=(n_0(g),n_1(g),\ldots,n_{n-1}(g)),
 \]
where \mbox{$ n_j(g)=\left|\{v\in V(g): s(v)=j,\, 0\le j\le n-1 \}\right|.$}
For example, the graphs in Figures~\ref{fig:SameDegree-a} and \ref{fig:SameDegree-b} both have shell distribution $(0,8,0,0,0,0,0,0)$.
The graphs in Figures~\ref{fig:SameDegree-c} and \ref{fig:SameDegree-d} have shell distributions $(0,0,8,0,0,0,0,0)$ and $(0,0,4,4,0,0,0,0)$, respectively.  These graphs illustrate the fact that the degree and core structures of a graph are not obtainable from one another.
Graph $g$ of Figure~\ref{fig:core-example} has shell distribution $(1,5,5,0,6,0,0,0,0,0,0,0,0,0)$.
\begin{figure}[h]
\begin{subfigure}[b]{.21\textwidth}
\centering
\begin{tikzpicture}
  [scale=.2,auto=center,every node/.style={circle,fill=black},inner sep=1.5pt]
  \path[use as bounding box] (0,-2) rectangle (12,6);
  \node (n1) at (0,0)  {};
  \node (n2) at (4,0)  {};
  \node (n3) at (8,0)  {};
  \node (n4) at (12,0)  {};
   \node (n5) at (0,4)  {};
  \node (n6) at (4,4)  {};
  \node (n7) at (8,4)  {};
  \node (n8) at (12,4)  {};
  \foreach \from/\to/\weight in {n1/n2/1,n2/n3/1,n3/n4/1,n2/n6/1,n4/n8/1,n5/n6/1,n6/n7/1}
    \draw (\from) --(\to);
 \end{tikzpicture}
  \vspace{0.45in}
\subcaption{\label{fig:SameDegree-a}
Vertices have degrees $1$, $2$, and $3$.}
 \end{subfigure}
 \quad
 \begin{subfigure}[b]{.21\textwidth}
\centering
\begin{tikzpicture}
  [scale=.2,auto=center,every node/.style={circle,fill=black},inner sep=1.5pt]
  \path[use as bounding box] (0,-2) rectangle (12,6);
  \node (n1) at (0,0)  {};
  \node (n2) at (4,0)  {};
  \node (n3) at (8,0)  {};
  \node (n4) at (12,0)  {};
   \node (n5) at (0,4)  {};
  \node (n6) at (4,4)  {};
  \node (n7) at (8,4)  {};
  \node (n8) at (12,4)  {};
  \foreach \from/\to/\weight in {n1/n2/1,n3/n4/1,n5/n6/1,n7/n8/1}
    \draw (\from) --(\to);
\end{tikzpicture}
 \vspace{0.45in}
 \subcaption{\label{fig:SameDegree-b}
 All eight 
 vertices have degree $1$.}
\end{subfigure}
\quad
\begin{subfigure}[b]{.21\textwidth}
\centering
\begin{tikzpicture}
  [scale=.2,auto=center,every node/.style={circle,fill=black},inner sep=1.5pt]
  \path[use as bounding box] (0,-4) rectangle (12,6);
  \node (n1) at (0,0)  {};
  \node (n2) at (4,0)  {};
  \node (n3) at (8,0)  {};
  \node (n4) at (12,0)  {};
   \node (n5) at (0,4)  {};
  \node (n6) at (4,4)  {};
  \node (n7) at (8,4)  {};
  \node (n8) at (12,4)  {};
  \foreach \from/\to/\weight in {n1/n2/1,n3/n4/1,n5/n6/1,n7/n8/1,n1/n5/1,n4/n8/1,n2/n6/1,n3/n7/1,n1/n6/1,n3/n8/1}
    \draw (\from) --(\to);
 \end{tikzpicture}
\subcaption{\label{fig:SameDegree-c}
  All vertices belong to the $0$-core, $1$-core and $2$-core. Higher cores are empty.}
 \end{subfigure}
 \quad
 \begin{subfigure}[b]{.21\textwidth}
\centering
\begin{tikzpicture}
  [scale=.2,auto=center,every node/.style={circle,fill=black},inner sep=1.5pt]
  \path[use as bounding box] (0,-4) rectangle (12,6);
  \node (n1) at (0,0)  {};
  \node (n2) at (4,0)  {};
  \node (n3) at (8,0)  {};
  \node (n4) at (12,0)  {};
   \node (n5) at (0,4)  {};
  \node (n6) at (4,4)  {};
  \node (n7) at (8,4)  {};
  \node (n8) at (12,4)  {};
  \foreach \from/\to/\weight in {n1/n2/1,n3/n4/1,n5/n6/1,n7/n8/1,n1/n5/1,n4/n8/1,n2/n6/1,n3/n7/1,n1/n6/1,n2/n5/1}
    \draw (\from) --(\to);
\end{tikzpicture}
 \subcaption{\label{fig:SameDegree-d}
 All vertices are in $k$-core for $k=0,1,2$, but $4$ of the vertices are also in the $3$-core.}
\end{subfigure}
\caption{\label{fig:SameDegree}
The graphs in (a) and (b) have the same core structure but different degree structure. The graphs in (c) and (d) have the same degree structure but different core structure.}
\end{figure}
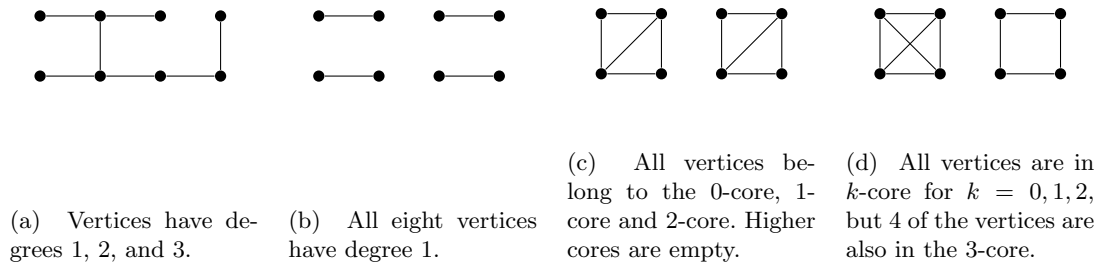

Finally, the \emph{degeneracy} of a graph $g\in \G_n$, denoted by $\dgen(g)$, is the index of the largest nonzero entry in the shell distribution vector $n_S(g)$. In other words, the degeneracy of a graph is the maximum index of a non-empty shell.
Thus we may define the following
subset of the set of simple $n$-vertex graphs $\G_n$:
\[\G_{n,m} =\{g\in\G_n: \dgen(g) = m\}.\]

\section{The shell distribution ERGM}
\label{sec:distributionModel}

A natural way to model random graphs using their core structure is to embed summaries of their core structure
 in the exponential random graph model (ERGM) framework.
  In what follows, we define a family of ERGMs   using  one such summary, namely the shell distribution, as a sufficient statistic.

Let $G=g$ be an instance of a random graph from the set $\G_n$. Partitioning  the vertex set of $g$ according to the shell indices implies that the probability of observing $g$ is

\begin{align}
P(G=g;p)=(\varphi(p))^{-1}\prod_{j=0}^{n-1}p_j^{n_j(g)},\label{eqn:graphprob}
\end{align}
where
$p_j\in(0,1)$ is the parameter that represents the propensity of shell $j$ to have vertices in it,
$p=(p_0, p_1, \ldots, p_{n-1})$ is the parameter vector,  integers $n_j(g)$ are the components of the shell distribution vector $n_S(g)$ as defined above, and $\varphi(p)$ is the partition function.
[One may also think of $p_j$ as representing the \emph{attractiveness} of shell $j$.]
Note that a feature of the model is that there is no dyad independence assumption.
Equation~\eqref{eqn:graphprob} is a most direct way to define an ERGM based on the shell distribution.
One can easily see that it can be written in exponential family form (see Appendix~\ref{sec:OldModel}) and allow us to take advantage of various good properties of exponential families.

It turns out, however, that specification \eqref{eqn:graphprob} of the model has many undesirable properties, common to other ERGMs \citep{RFZ:mle09}; details are given in  Appendix~\ref{sec:OldModel}.
There are several ways to avoid these issues that arise from specifying the model as in Equation~\eqref{eqn:graphprob}; one such way is to add an additional parameter to the model as follows.
We restrict the support of the model to the set $\G_{n,m}$  of all simple graphs whose degeneracy is
 equal to $m$. 

\begin{align}\label{eqn:multinomialModel}
P(G=g;p,m) =
	\begin{cases}
		(\varphi(p))^{-1}\prod_{j=0}^{m}p_j^{n_j(g)} &\mbox{if }g \in \mathcal G_{n,m},\\
		0 &\mbox{otherwise,}
	\end{cases}
\end{align}
where
$$\varphi(p) = \sum_{g \in \mathcal G_{n,m}} \prod_{j=0}^{m}p_j^{n_j(g)} $$
is the normalizing constant (partition function).
Equation~\eqref{eqn:multinomialModel} defines a multinomial-like distribution over the partition of nodes induced by the shell distribution.
By limiting degeneracy, the model has a significantly reduced  number of parameters, which offers an additional advantage in estimation over the more general model.

For each fixed value $m$ of  
 degeneracy, the model defined by Equation~\eqref{eqn:multinomialModel}  is an ERGM supported on the subset of graphs $\mathcal G_{n,m}$.
We have thus defined a  family of models  with  parameters $p$ and $m$, where $p=(p_0,\dots,p_m) \in \Delta_{m+1}$ and $m \in \{0,\ldots, n-1\}$.  It is a union of ERGMs, one for each distinct value of $m$.

For the remainder of the paper, this support restriction is assumed to be present and made implicit, unless otherwise mentioned, to ease notation. The dimension of the parameter space is $m+1$ and is a function of the parameter $m$.

\begin{rmk} \label{rmk:degen}
In this paper, we will treat $m$ as fixed and known. When fitting the model to real networks, $m$ will be selected by setting it equal to the degeneracy of the observed graph, assuming the sample size $N=1$ as is most common in applications.
Estimating $m$ and fitting  the shell ERGM when $N>1$ and the observed graphs have distinct degeneracy values is an open question.
The choice of fixing $m$ rather than treating it as an estimable parameter is both reasonable and warranted. The degeneracy of a graph is an important metric that describes its sparsity and is easily calculable from the data. If the degeneracy is not fixed, the large majority of our parameters will not be estimable as the observed graphs are expected to be sparse (real networks usually are), with observed degeneracy much smaller than $N$, see also \ref{subsec:estimation}.
Moreover, simulations show that allowing $m$ to be different from 
 the observed degeneracy leads to a poorly behaved model,  as explained in Section \ref{sec:appendixDegenerate}. Intuitively, having $p_i>0$ for large shell indices $i$ ensure that large-index shells  attract most nodes.
\end{rmk}

In order to express this model in exponential family form, define the set of natural parameters $\theta_i = \log (p_i/p_{m})$. Note that by definition, $\theta_{m} = 0$,  so there are $m$ linearly independent parameters; we  will thus denote by
$\theta = (\theta_0, \ldots, \theta_{m-1})$  the vector of natural parameters. The shell distribution ERGM can now be written in the following form:

\begin{align}
P(G=g)=\exp\left\{\sum_{j=0}^{m-1}n_j(g)\theta_j-\psi(\theta)\right\},
\end{align}
where $\psi(\theta)$ is the log-partition function (or the log normalizing constant),  given by
\begin{align}\label{eq:shellERGM}
\psi(\theta) = \log \sum_{g \in \mathcal G_{n,m}} \exp\left\{\sum_{j=0}^{m-1}n_j(g)\theta_j\right\}.
\end{align}

The $m$-truncated shell distribution
 $( n_0(g), \ldots, n_{m-1}(g) )$ is a minimal sufficient statistic of the model. The natural parameter space is
\begin{align}
\Theta = \{ \theta \in \mathbb R^m: \psi(\theta) < \infty\} =\mathbb R^m.
\end{align}
 Given this model specification, the overarching objective is to use it to perform statistical inference.  
  However, as is usually the case for ERGMs, evaluating the log-partition function above is intractable for any reasonably sized $N$.  This will affect the computation of the maximum likelihood estimator (MLE), requiring one to resort to MCMC methods, as well as testing model fit.   In the remainder of this paper, we study three important aspects of these problems.  First, both  MLE computation and model fitting depend on our ability to sample from the model with a given parameter value. To this end, we provide an MCMC algorithm for sampling from the model, summarize the results of several simulations, and provide an interpretation of the model parameters and the sampling distribution of realizable graph shell structures.  Second, from the theory of exponential families, we know that the MLE is unique if it exists. But the question of existence is not often easy to address; we solve it here for the shell distribution model.
Finally, testing model fit necessitates the ability to sample from the \emph{fibers} of the model, that is, the subspaces of $\G_n$ with given fixed values of the shell distribution.  We provide an algorithm for performing this task. We begin with theoretical considerations, then proceed to simulation results.

\subsection{Sample space restriction and degeneracy of real-world networks} 

In ERGMs, sample space restriction leads to an improvement in the properties of the conditional model and estimation algorithms, as shown in  \cite{snijders2002conditional}, \cite{snijders2002markov}. A usual approach is to condition on the degree sequence,  maximum degree, or degree distribution, etc. In contrast, we are conditioning on the observed degeneracy of the graph. This is more robust than conditioning  on the degree, as we are  allowing the degrees to be somewhat free but still controlling sparsity in another way. 

Degeneracy of real networks tends to be small relative to the number of nodes. A table illustrating this for the undirected graphs from the \href{ http://vlado.fmf.uni-lj.si/pub/networks/data/}{Pajek} collection of datasets \citep{pajek} is included  below. 
\begin{table}[h]
	\centering
	\begin{tabular}{|c|c|c|c|>{\centering\arraybackslash}p{80mm}|}\hline
		Network Dataset & \#Nodes&\#Edges & Degen.&Shell Distribution\\ \hline
		Scotland  & 244&256&4&$(16,26,183,7,12)$\\ \hline
		Geom & 7343&11898&21&$(1185,2218,1714,1023,503,248,122,126, \newline 34,27,20,52,0,1,7,14,17,0,0,0,0,22)$\\ \hline
		NDyeast &2114&2277&5& $(244,1199,478,169,18,6)$\\\hline
		NetScience &1589&2742&19&$(128,320,390,281,223,89,21,60,27,30,0,0,0,0,0,0,0,0,0,20)$\\\hline
		USpowerGrid&4941&6594&5&$(0,1588,3122,195,24,12)$\\\hline
		Erd\H{o}s&6927&11850&10&$(0,4780,954,466,258,179,113,73,49,17,38)$\\\hline
	\end{tabular}
\end{table}
\noindent

Observe  that the degeneracy of the graph is allowed to grow as the number of nodes grows, but is expected to be significantly smaller than $n$ in real-world networks. 

\smallskip 
\section{Inference and implementation of the shell distribution ERGM}
\label{sec:mcmc}

\label{sec:MCMCSampling}

Many inference problems associated with ERGMs require generating random samples from the model at a fixed parameter value. In particular, problems such as computing an MLE using Monte Carlo methods (\citet{snijders2002markov}), sampling from the posterior distribution of the parameters (\cite{caimo2011bayesian}) and exploring the space of graphs that have high probability under the model each require random samples from the model. In this section, we present a commonly used MCMC algorithm to sample graphs from the shell distribution ERGM and use this algorithm to obtain maximum likelihood estimates and to understand 
the properties of random graphs that arise from the shell distribution  ERGM. 

\paragraph{Sampling from the shell distribution ERGM:}As is the case with most ERGMs, sampling from the shell distribution ERGM is intractable and we need to resort to Markov chain Monte Carlo (MCMC) schemes. We use a Metropolis-Hastings algorithm  with a tie-no-tie proposal (see \cite{caimo2011bayesian}) to generate graphs from the model. At each iteration, the algorithm proposes a graph $g'$ from the current state $g$ and decides to accept it with probability
 \begin{align}
 \label{mhprob}
 	\min\left(1,\frac{P(g')\cdot P(g'\rightarrow g)}{P(g)\cdot P(g\rightarrow g')}\right)
 = \min \left( 1,\prod_i p_i^{n_i(g')-n_i(g)}\cdot \frac{P(g'\rightarrow g)}{P(g\rightarrow g')}\right),
\end{align}
where $\{g \rightarrow g'\}$ denotes the event that the Markov chain moves from $g$ to $g'$. Note that when the proposed graph $g'$ has degeneracy not equal to $m$, by definition of the model, $P(g')=0$, hence the acceptance probability is $0$. 

A simple proposal distribution that is commonly used for proposing new graphs in the Metropolis framework is to randomly select a dyad and swap it.
 However, during experiments, we found that this leads to Markov chains with poor mixing properties. Instead, we use a 
 ``tie-no-tie'' (TNT) proposal, also used in \cite{caimo2011bayesian}.
 At each iteration, the TNT proposal randomly chooses between the set of edges and non-edges, and then swaps   a randomly chosen dyad within the selected set.
But  this proposal is not symmetric: Let $\pi$ be the probability of choosing the set of edges, $ne(g)$ be the number of non-edges in $g$ and $e(g)$ be the number of edges in $g$. Then the Hastings ratio $\frac{P(g'\rightarrow g)}{P(g\rightarrow g')}$ is determined  as follows:
\begin{align}
\frac{P(g' \rightarrow g)}{P(g \rightarrow g')} = \begin{cases} 
\frac{\pi}{1 - \pi}\frac{ne(g)}{e(g)+1}, & \mbox{if } g' \mbox{ is obtained from } g \mbox{ by adding an edge} \\ 
\frac{1 - \pi}{\pi}\frac{e(g)}{ne(g)+1}, & \mbox{if } g' \mbox{ is obtained from } g \mbox{ by removing an edge.} 
\end{cases}
\end{align}
	
\begin{rmk}
	Computing the acceptance probability using equation \ref{mhprob} requires one to compute the 
	so-called vector of ``change statistics'' $\{n_i(g') - n_i(g)\}, i = 1, \ldots, n$ at each step, see \cite{hunter2006inference}. For many existing ERGMs, the change statistics can be computed locally, i.e without resorting to computing the sufficient statistics for proposed network $g'$. However, this is not the case for the shell distribution  as it is a \emph{global} sufficient statistic. In order to compute the change statistics, we need to recompute the shell distribution for the proposed network $g'$ at each step of the Markov chain. This increases the computational complexity of the algorithm, even though one can compute the shell distribution in linear time.\end{rmk}
\subsection{Estimating the parameters of the shell distribution ERGM:}
\label{subsec:estimation}
A natural starting point to estimate parameter values $\theta$ and $m$ using a real network is by either (a) using their observed counterparts, (b) by using a maximum likelihood estimate. We will discuss these two estimating methods for both $\theta$ and $m$.  Estimation of $m$ is tricky, as it represents the model dimension, and we observe only one graph. Also for any observed graph, allowing $m$ to be different from the observed degeneracy leads to many undesirable properties of the resulting model. We explain this issue at length in Section \ref{sec:appendixDegenerate}. Thus for simulation studies based on real networks we fix $m$ to be the observed degeneracy.

Estimation of $\theta$ is more involved. One can estimate $\theta$ naively by using the empirical shell distribution and setting $\hat{\theta}_j = n_j/n$, or one can use a more principled likelihood-based estimator (such as an MLE or a Bayes estimate). It turns out that using the observed shell distribution as an empirical estimate leads to a poor (or uninteresting) parameter estimate - in particular, networks sampled from the empirical estimate do not resemble the observed network. Namely, the model puts most of  
its  mass on graphs with all nodes in the largest possible shell (see also Sections~\ref{equal} and~\ref{sec:Sampson}).  
On the other hand, computing an MLE of $\theta$ from the observed network is intractable due to the  normalizing constant $\psi(\theta)$ given in Equation~\eqref{eq:shellERGM}. Maximizing the likelihood requires  the repeated use of Markov Chain Monte Carlo sampling, as described below, see also \cite{hunter2006inference} and references therein. Bayesian estimates are also intractable due to two normalizing constants, see \cite{caimo2011bayesian} for more details.

We use Markov chain Monte Carlo MLE (\cite{geyer1992constrained,snijders2002markov}) to estimate $\theta$. For $t=0,1, \ldots$, let  $\theta^t$ be the parameter estimate at iteration $t$. We  estimate the ratio of the intractable normalizing constant $\frac{\psi(\theta)}{\psi(\theta^t)}$ using samples from $\theta^t$ obtained by the Markov chain algorithm presented earlier. Specifically, let $g_1, \ldots, g_B$ be a random sample from the model $\theta^t$, then 

$$\frac{\psi(\theta)}{\psi(\theta^t)}  \approxeq \frac{1}{B}\sum_{b=1}^B{\exp\left\{(\theta-\theta^t)n_S(g_b)\right\}}.$$
Then, $\theta^{t+1}$ is estimated by maximizing the estimated log-likelihood, given by 
$$\hat{l}(\theta,\theta^t) = (\theta-\theta^t)n_S(g_{obs}) - \log \frac{\psi(\theta)}{\psi(\theta^t)}  $$
and the process is repeated until convergence, see \cite{hunter2006inference} for more details. 

Estimation of the normalizing constant requires a good initial value $\theta_0$ (\cite{hunter2006inference}). We use a heuristic grid search to obtain a good starting point that is close to the MLE, where closeness to the MLE is evaluated by checking if the empirical version of the following moment equation holds: 

$$E_{\hat{\theta}}[n_s(g)] = n_s(g_{obs}),$$ where $g_{obs}$ is the observed graph and $\hat{\theta}$ is an MLE.

The behavior of the MCMC-MLE estimator depends on the choice of a good starting point $\theta_0$. For the current simulations, we use a heuristic starting point, but one could also consider the step length algorithm in \cite{hummel2012improving} to find a good starting point close to the MLE. 

\paragraph{What do graphs from the shell distribution ERGM look like?} 
We use the MCMC algorithm described above to explore the structure of random graphs generated by the model for fixed and estimated parameter values. In particular, for a fixed choice of parameters $\theta$ and $m$ of the shell distribution ERGM, we explore the space of graphs that have high probability mass under the model by sampling a large number of graphs $\{g_b\}_{b=1}^B$ using the MCMC algorithm. We use these samples to find out what features of any given network can be captured by modeling its 
core structure through the shell distribution ERGM. 
In the simulation studies below, we employ two types of parameter values to simulate graphs - known fixed parameters and parameters estimated from a real-world network. For the known parameters, we always use degeneracy $m=3$. Parameter estimates based on real world networks are obtained using a combination of a heuristic grid search (to initialize the MCMC MLE algorithm) and MCMC MLE. 
To explore the sampled space of graphs, we summarize the distribution of the sampled graphs $\{g_b\}$ by using several summary statistics: boxplots of the degree distribution and shell distributions, and histograms of number of edges, two stars, and triangles, centrality, size of largest shell and size of the innermost shell. When the parameters are estimated using a real world network, we also compare the distribution of these summary statistics with the corresponding observed statistic. It may be tempting to use this comparison as a way to assess the goodness of fit of the model, however, one must exercise caution:  

\begin{rmk}
	It is important to note that comparing the sampling distribution of summary statistics with the observed values is not a formal goodness-of-fit test of the model, but instead a heuristic approach to evaluate how well the model fits the  data. It follows along the lines the goodness-of-fit testing proposed for more general ERGMs in \cite{hunter2008goodness}.
	Ideally, one should be able to either derive the asymptotic distribution of any test statistic or, since in this case we usually observe a single network, perform an exact test. However, doing so requires several important steps, foremost, a good choice of a test statistic that can play the role of a generalized goodness-of-fit statistic. In case of, say, hierarchical log-linear models for contingency tables, one can use the chi-square statistic, and sample from the conditional distribution given the observed sufficient statistic to approximate the exact distribution of $\chi^2$. In case of this ERGM, however, we do not have at our disposal such a statistic that can reliably `measure' the distance of the observed network from the expected network. The main obstacle is that the dyads are not independent in this model,  unlike the case of hierarchical models in which  cells in the  contingency table (arising from the incidence matrix) are independent.
	To this end, we follow the generally used strategies for ERGMs and report the sampling distributions of various complementary network statistics, such as the number of edges and the number of triangles. For completeness, we explore the distribution of these statistics when conditioning on the sufficient statistics in Section~\ref{sec:structural}.
\end{rmk}
\subsection{Example 1: Various fixed Shell probabilities }\label{equal}
In this section, we study the properties of the shell distribution ERGM by simulating graphs from various fixed parameters. We set $m=3$, $n=18$ and consider two models: 
\begin{enumerate}
	\item Equal attractiveness, i.e., $p_i  = \frac{1}{4}$ for all $i$; % \forall i$; 
	\item Decaying attractiveness, i.e, $p_i \propto e^{-i}$ for all $i$.   
\end{enumerate}

\begin{figure}[t]
	%\vspace{.3in}
	\begin{minipage}[b]{0.44\textwidth}
		\centering
		\includegraphics[scale=0.42]{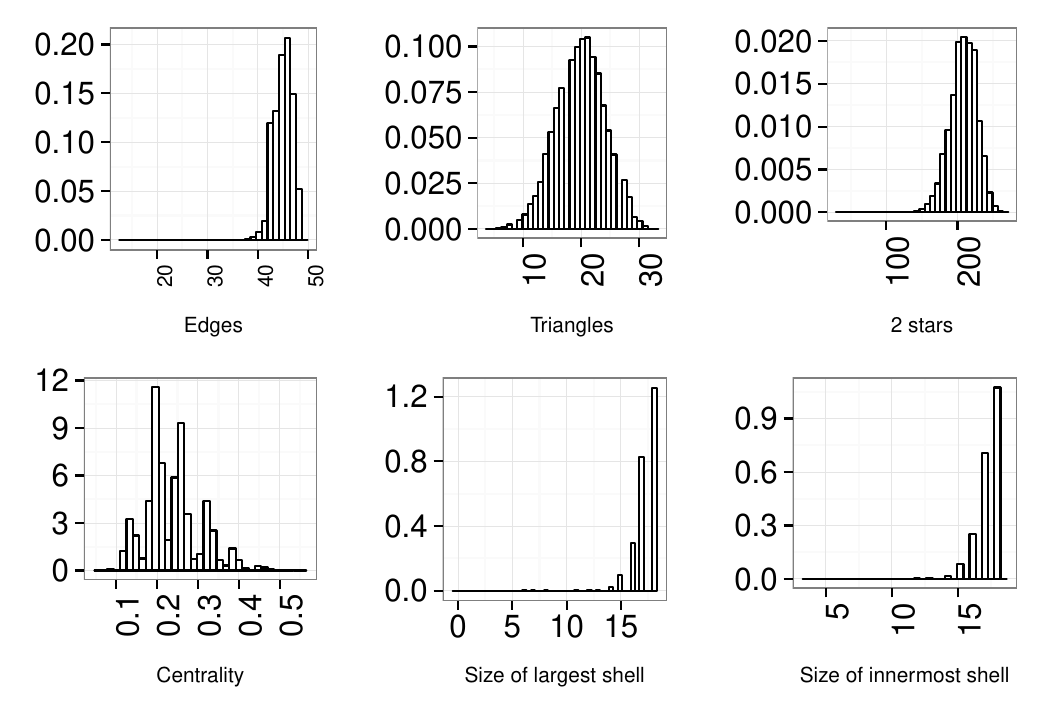}
		\caption{Sampling distributions of summary statistics from the \emph{Equal Attractiveness} model}\label{fig:equal}
	\end{minipage}
	%\end{figure}
	%\begin{figure}[t]
	\quad\quad\quad\quad
	\begin{minipage}[b]{0.43\textwidth}
		\includegraphics[width=\textwidth]{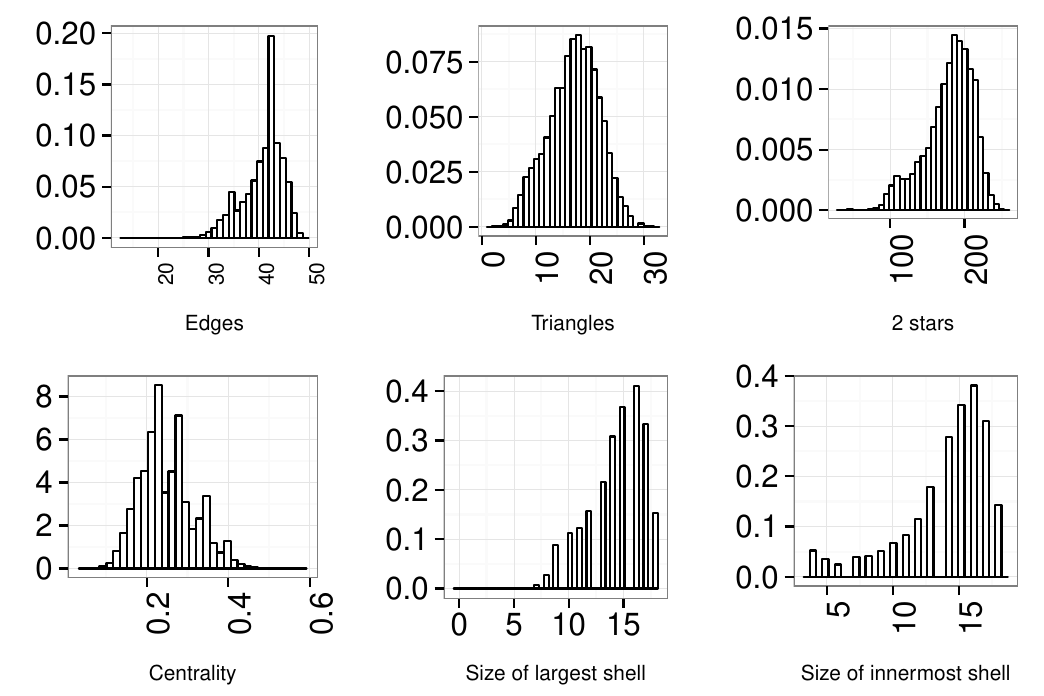}
		\caption{
			Sampling distributions of summary statistics from the \emph{Decaying Attractiveness} model }\label{fig:decay}
	\end{minipage}
	%\vspace{.3in}
\end{figure}

\emph{Equal attractiveness Model:} This model posits that every shell has equal attractiveness, i.e. $p_i = \frac{1}{4}$ for all $i$ and since $\theta_i=\log\frac{p_i}{p_m}$, it follows that $\theta = (0,0,0,0)$. 
Hence by definition, this model places a uniform mass over the set of all $3$-degenerate graphs. The sampling distribution of various summary statistics of graphs sampled from this model are shown in  Figure~\ref{fig:equal}. Note that even though the model posits that every shell has equal attractiveness a priori, the sampled graphs are such that most nodes tend to lie in the innermost shell which is shell $3$ in this case. This can be seen by the histogram of the size of the innermost shell in Figure \ref{fig:equal}.  There are at least three reasons for this behavior, the first one  related to the very definition of the shell index. Namely, the existence of higher-index shells in a graph  requires a certain minimum  number of nodes in it, and hence, a priori, higher shells have higher levels of natural ``attractiveness'', to which we refer as intrinsic graph-theoretic attractiveness. In this sense, the innermost shell is always the most attractive. 
Secondly, the model puts a uniform distribution on the space of all graphs, not on the space of all shell distributions. For example consider the $4$-truncated shell distributions $(0,0,0,18)$ and $(18,0,0,0)$: there are many  graphs realizing the former, yet exactly one graph realizing the latter, namely the empty graph. 
 Thus, the sampling distribution of the shell distributions is non-uniform. Finally, there is also an issue with the slow mixing of the Markov chain. Shell distributions with a large number of nodes in the higher-indexed shells are ``stable'' in the sense that adding or removing a single edge tends to leave the shell distribution unchanged. On the other hand, when most nodes are in lower index shells, adding or removing a few edges lead to large changes in the shell distribution.

 It is worth noting that the second and the third issue above are, in fact, related to each other and also to an issue  that arises naturally in  ERGMs in general.  Namely, ERGMs model random graphs,  not sufficient statistics, thus  a uniform distribution over the set of graphs is not a uniform distribution over the set of sufficient statistics one may care about. This is made evident by the current example: a uniform distribution over $3$-degenerate graphs induces a non-uniform distribution on the graph statistics such as number of triangles, number of edges, and $2$-stars.

\emph{Decaying Attractiveness Model:} The decaying attractiveness model posits that the attractiveness of each shell decays exponentially with its index, i.e. $p_i = c{e^{-i}}$, where $c$ is some constant. This model aims to overcome the problems imposed by the intrinsic graph-theoretic attractiveness of  the higher-index shells. Figure \ref{fig:decay} shows the sampling distributions of summary statistics of the samples from this model. The histogram of the size of the innermost shell has two modes, one at $16$ and a second one at $4$, suggesting a bimodal distribution. The histograms of number of two stars and the number of triangles are bimodal as well. 
	
\subsection{Example 2: Sampson monastery data}\label{sec:Sampson}

The Sampson dataset is a widely studied network of size $18$ that records interactions among a group of monks in a New England Monastery \cite{sampson1968novitiate} and their evolution over time. 
The first three time periods of the original Sampson data are commonly used (e.g., in the {\tt ergm} package) and often aggregated. The network at any of these three time periods, makes for an uninteresting second example from the point of view of shells: namely 
all nodes are in the same shell and of degeneracy~3 and we have already considered such networks in Section~\ref{equal}.  The aggregate network over the three time periods also has just about all nodes (all but 4) in the highest shell and of degeneracy 5. In order to obtain a more varied shell distribution as a case study to examine the model behavior,
we consider instead an arbitrary subgraph of the aggregate network; specifically, we use the upper triangular part of the adjacency matrix and symmetrize it. 
This undirected network is shown in Figure \ref{fig:sampsonNetwork}, color-coded by shells; it has $n = 18$ nodes, $e = 35$ edges and density of $0.23$. The observed degeneracy is $3$ and the observed $4$-truncated shell distribution is $(0,2,3,13)$; there are $3$ nonempty shells, and  
the innermost shell (shell $3$) contains the highest number of nodes ($13$).

To use this Sampson-derived network to  study the properties of the shell distribution ERGM, we set $m=3$ and use MCMC MLE to estimate the value of $\theta$. 
 Using a heuristic grid search, we found $\theta_0 = (2,1,1,0)$ to be a good initial estimate. The estimated MLE is $\hat{\theta}_{MLE} = (-7.95,  2.79,  0.91,0)$ which corresponds to $\hat{p}_{MLE}  = (0.00, 0.82, 0.13, 0.05)$.  Recall that $\theta_i = \log \frac{p_i}{p_m}$ and hence $\theta_i$ can be interpreted as the log-odds of attractiveness of shell $i$ relative to shell $m$. For this dataset, attractiveness  of shell $1$ relative to shell $3$ is almost $3$ times that of shell $2$, thus indicating that the network has a  rich periphery in the sense of \cite{RPFM14}. This can also be seen by noting that $\hat{p}_1 = 0.82$; recall that the $p_i$ can also be interpreted as the propensity of the $i$-th shell to have nodes in it beyond its intrinsic graph-theoretic attractiveness (as explained in  Section~\ref{equal}).

Next, using $m=3$ and the MLE estimate  $\theta = (-7.95,2.79,0.91,0)$, we simulated networks from the model using the MCMC algorithm presented earlier in this Section to study what properties of the network are captured by the model. One can think of these sampled graphs as samples from the posterior predictive distribution. Convergence of a 40,000-step Markov chain was verified using the usual diagnostics, such as trace plots and autocorrelation plots to ensure sufficient mixing. Figures \ref{fig:sampsonHistogram}, \ref{label-a}, \ref{label-b} summarize the results of the simulations.

\begin{figure}[t]
	%\vspace{.3in}
	\begin{minipage}[b]{0.42\textwidth}
		\centering
		\includegraphics[scale=0.42]{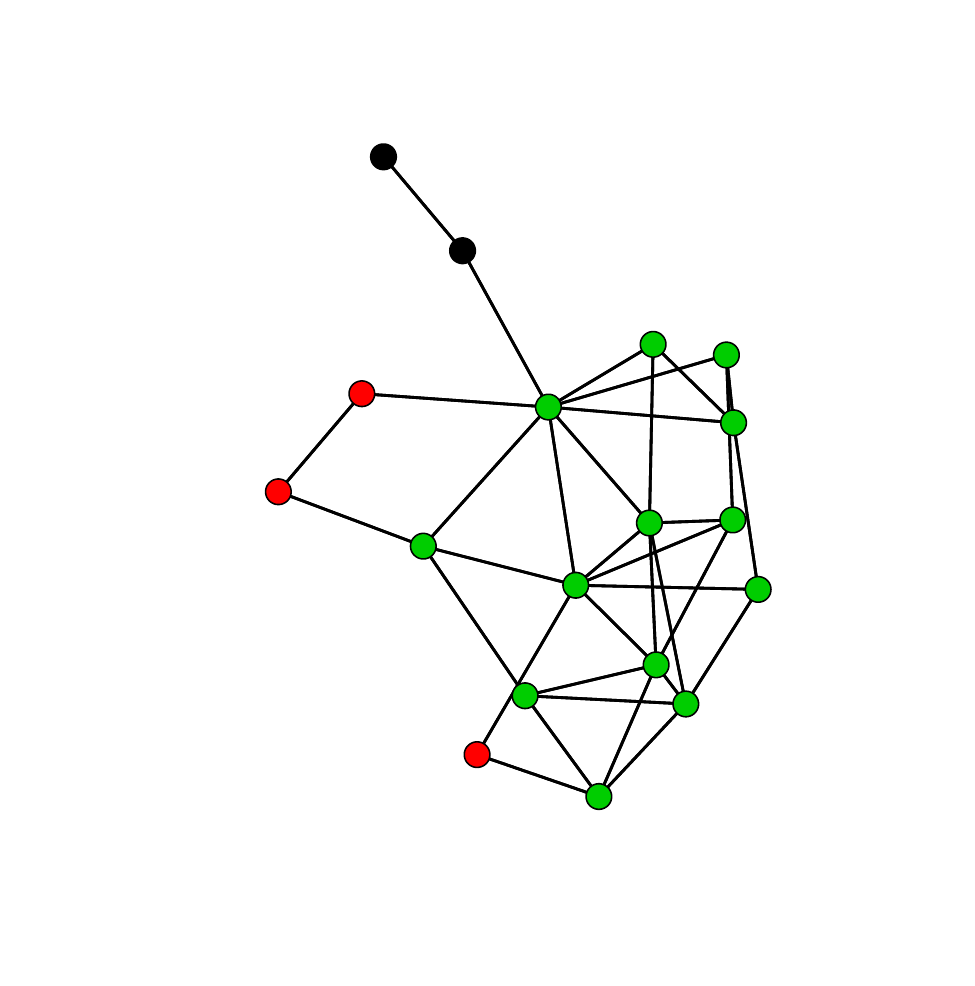}
		\caption{A subset of the Sampson Monastery Dataset: Nodes are colored according to their shell index: black is $1$, red is $2$, and green is shell index $3$.}\label{fig:sampsonNetwork}
	\end{minipage}
	%\end{figure}
	%\begin{figure}[t]
	\quad\quad\quad\quad
	\begin{minipage}[b]{0.42\textwidth}
		\includegraphics[width=\textwidth]{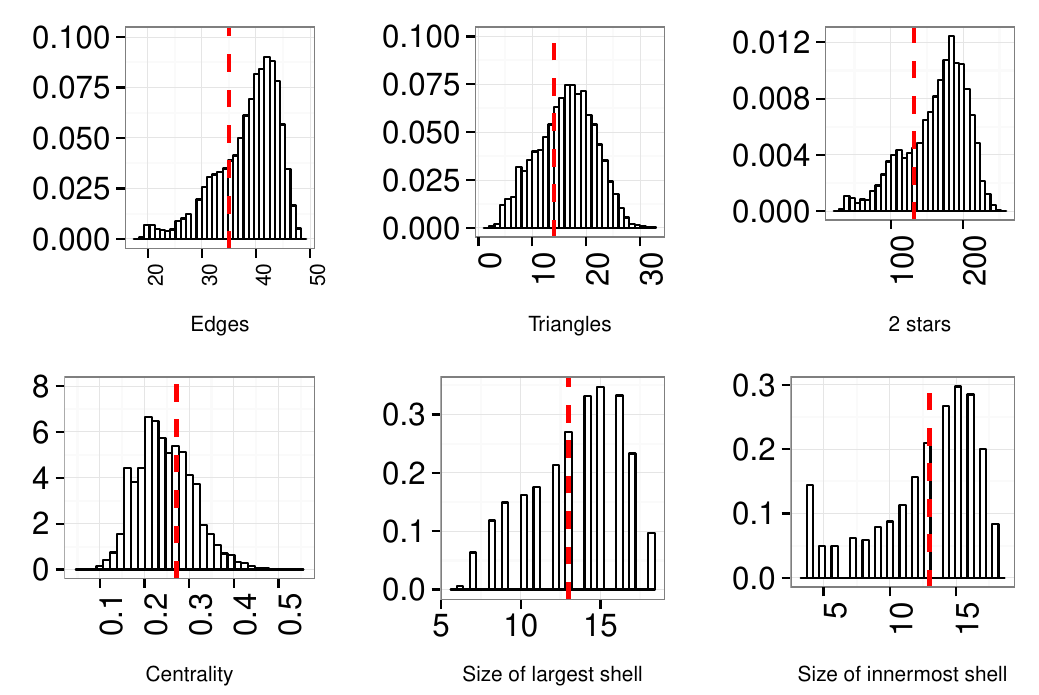}
		\caption{
			Sampling distribution of summary statistics from the model estimated from the  dataset in Figure~\ref{fig:sampsonNetwork}. The red dashed lines indicate the observed values of the statistics.}\label{fig:sampsonHistogram}
	\end{minipage}
	%\vspace{.3in}
\end{figure}

\begin{figure*}
\centering
\centering
\begin{minipage}[b]{0.40\textwidth}
\includegraphics[width=\textwidth]{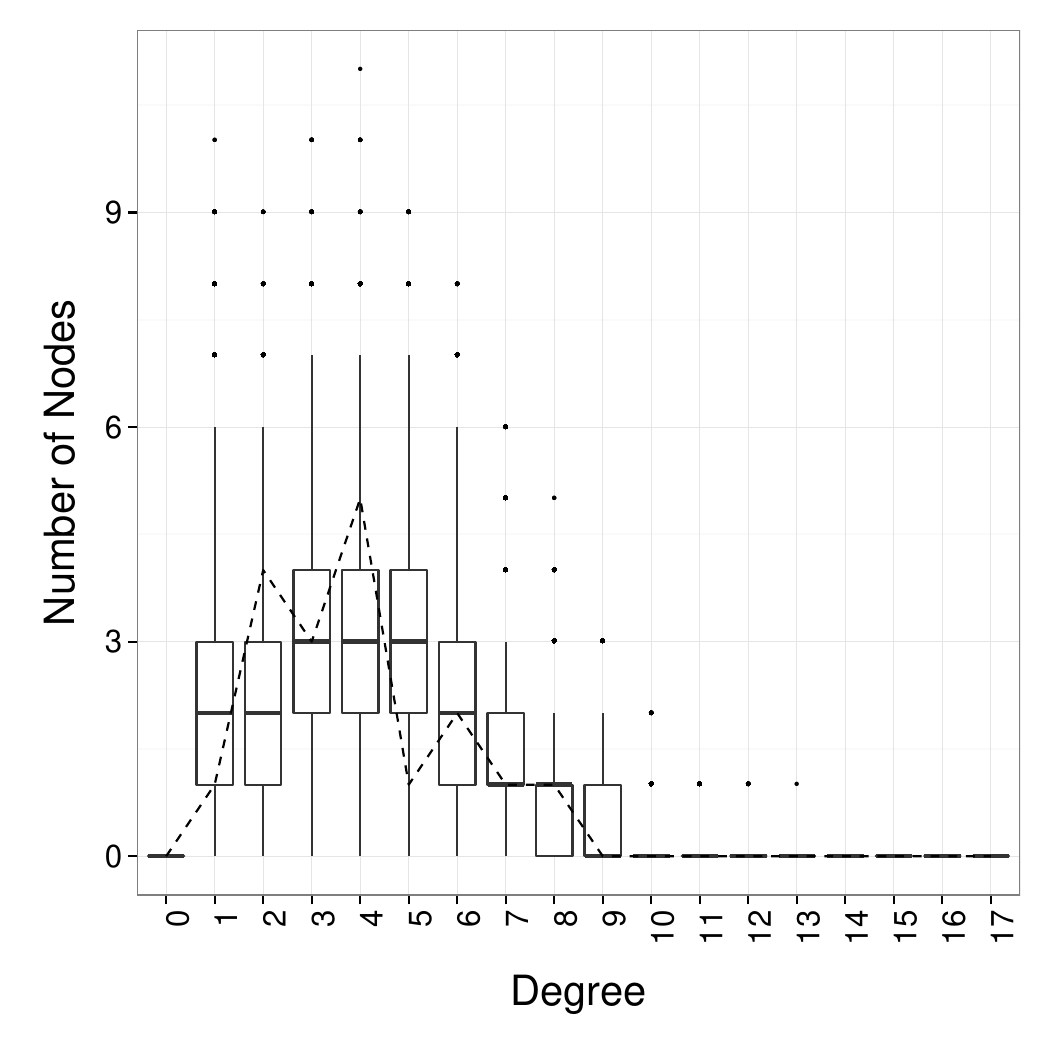}
\caption{Degree Distribution}\label{label-a}
\end{minipage}\qquad
\begin{minipage}[b]{0.40\textwidth}
\includegraphics[width=\textwidth]{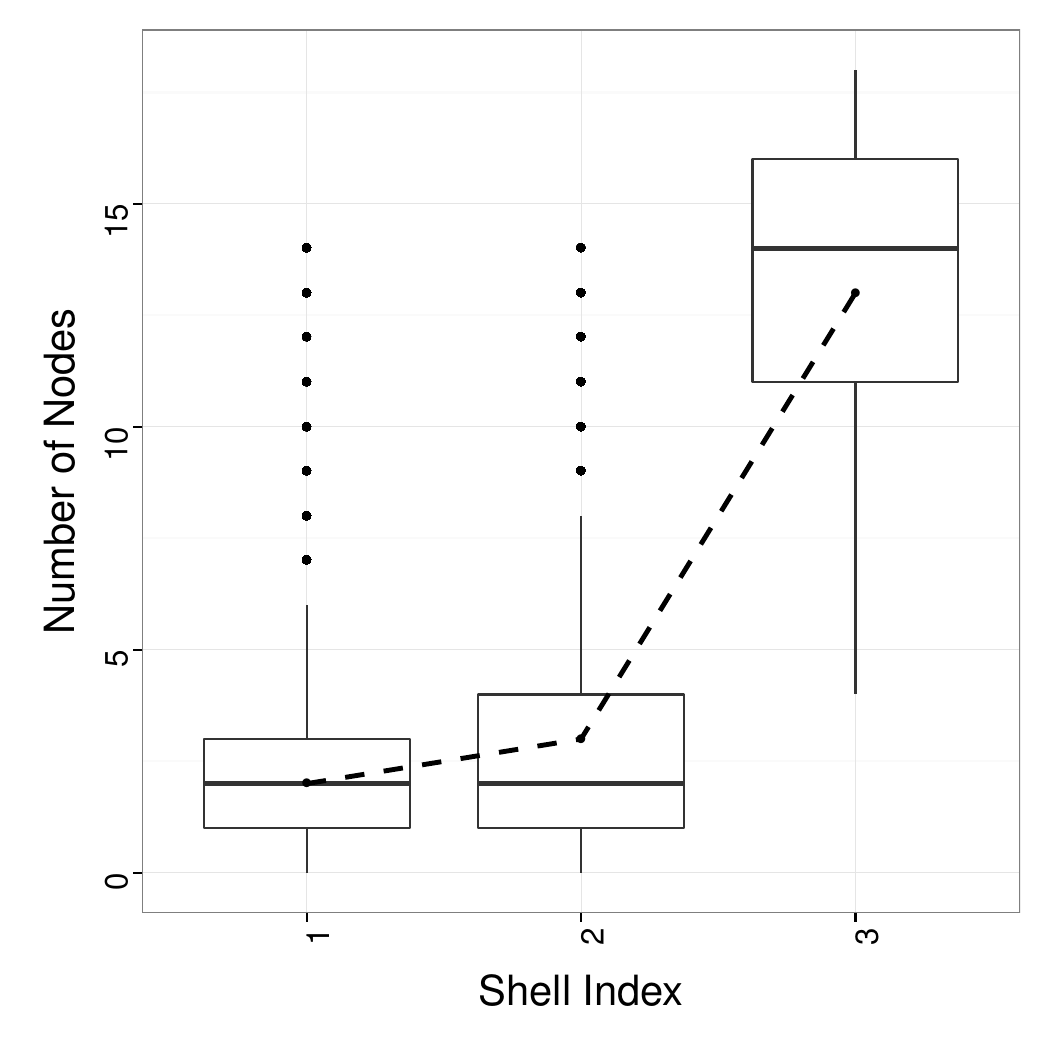}
\caption{Shell Distribution}\label{label-b}
\end{minipage}
\caption{Box plots of degree distributions and shell distributions for the shell distribution model estimated from Sampson data. The dashed lines represent the observed  distributions.}
\end{figure*}
Specifically, Figure \ref{fig:sampsonHistogram} shows  the sampling distribution of various summary statistics in the form of a histogram and compares them with the observed values. Several interesting results emerge.  The sampling distribution of the summary statistics are all unimodal and very close to the observed statistic shown by the red line. Notice that the histogram of triangles is centered around the observed value, thus the shell distribution model captures triadic effects quite well, at least in this small example. We would like to draw a comparison with degree-based models which do not capture triadic effects, by definition. It is widely believed that the centrality of a network is related to its core distribution, and the histogram of centrality provides additional support of this hypothesis. The distribution of the size of the largest shell is also captured by the model. However, the sampling distribution of number of edges suggests that the observed number of edges is much smaller than what we expect under the model. This may be due to the fact that the model has a bias towards graphs with 
 higher-index shells (innermost cores), and these shells tend to be densely connected. A similar situation is true for the number of two-stars. The sampling distribution of the size of the innermost shell indicates that it can have anywhere from $5$ to $18$ nodes, with two modes at $15$ and $16$; compare this with the observed number of $13$ nodes in shell $3$. We also consider various shell distributions visited by the Markov chain. The top 10 most frequently visited shell distributions are given in Table \ref{tab:top10}. 
 \begin{table}
 \begin{center}
 	\caption{The top 10 visited shell distributions}
 	\label{tab:top10}	
 	\begin{tabular}{|c|l|}
 		\hline
   Shell Distribution & Density (in $\%$) \\
   \hline
0.1.1.16	&	5.95	\\
0.0.1.17	&	5.79	\\
0.1.2.15	&	5.22	\\
0.0.2.16	&	4.54	\\
0.2.1.15	&	4.22	\\
0.0.0.18	&	3.88	\\
0.1.3.14	&	3.64	\\
0.2.2.14	&	3.63	\\
0.1.0.17	&	3.54	\\
0.0.3.15	&	2.89	\\

 		\hline 
 	\end{tabular} 
 \end{center}
 \end{table}

Figures \ref{label-a} and \ref{label-b}  show the box plots of  degree and  shell distributions, respectively, of the sampled graphs, and include
 the observed degree and shell distributions as dotted lines.  Note that the sampling distribution of  degree distributions is quite different from that of  shell distributions, showing that the shell distribution model captures features that go beyond the degrees, and justifying our initial motivation for constructing the model.
 In addition, the sampling distribution of the shell distribution is concentrated around the observed shell distribution. This is to be expected: as we used the observed shell distribution to estimate the model, it serves as a check that the MLE of $\theta$ using MCMC MLE is indeed a good estimate. Recall that another definition of the MLE is the following: If $\hat{\theta}$ is an MLE, then, $E_{\hat{\theta}}[n_s(g)] = n_s(g_{obs})$. Figure \ref{label-b} serves as a visual confirmation of this equation. In fact, the observed shell distribution is $n_s(g_{obs}) = (0,2,3,13)$ and the estimate of the expected shell distribution (based on the MCMC samples from $\hat{\theta}$) turns out to be $\hat{E}_{\hat{\theta}}[n_s(g)] = (0.00,  2.29,  3.06, 12.66)$. 
 Finally, even though the  general trend in the observed degree distribution is captured by the model, as suggested by Figure \ref{label-a}, there is a substantial deviation between the observed degree distribution and the one suggested by the model. This reinforces the observation that the degree distribution and shell distribution capture different aspects of the Sampson network, and the shell distribution ERGM captures properties of the network beyond the degrees. In fact, it is well-known that degree-based models have independent dyads, whereas the shell distribution ERGM does not. This is further evidenced by Figure \ref{fig:sampsonHistogram}.

\section{A sampling algorithm for generating graphs with a given shell distribution} 
\label{sec:samplingSequence}

In the last two decades, there have been several contributions in the graph theory and computer science literature on computing cores decompositions. Given the wide-ranging application of cores, a natural problem that arises is to find an algorithm that randomly generates graphs with a given core structure.  Such an algorithm is presented in \cite{BGGKW07} for graphs with additional restrictions on the number of edges between pairs of shells.

This section provides a simple algorithm (Algorithm~\ref{alg:ShellSeq}) for sampling the space of graphs with a given shell distribution (sometimes called the \emph{fiber} of that distribution), such that any graph has positive probability of being constructed (Theorem~\ref{thm:PosProbEveryG}).  This is an independent sampler, not a Markov chain.
Simulations indicate good performance in terms of discovering new graphs at a fast pace. While the true sampling distribution is not known, our experiments show that reasonably long runs will give good estimates.
{\scriptsize
\begin{algorithm}[h]
\LinesNumbered
\DontPrintSemicolon
\SetAlgoLined
\SetKwInOut{Input}{input}
\SetKwInOut{Output}{output}
\Input{a graph $g$}
\Output{its shell sequence $s(g)=(s_1,\ldots,s_n)$}
\BlankLine
Initialize $s^*=0$.\;
Repeatedly remove vertices of degree at most $s^*$ in $g$,
incrementing $s^*$ by 1 if no eligible vertices remain in $g$; quit when $g$ is empty.
The shell index of each vertex is the value of $s^*$ when it was deleted.
\caption{Compute Shell Sequence}
\label{alg:GetShellSeq}
\end{algorithm}
}

For convenience, we restate the basic algorithm for producing the shell sequence of a graph as Algorithm~\ref{alg:GetShellSeq}.  There is no need to implement it, since the linear-time algorithm from \cite{BZ-coresDecompAlgo2003}
 is already implemented as
the {\tt graph.coreness} function from the \cite{igraph} {\tt igraph} package in {\tt R}.

Note that the order in which the vertices of $g$ are deleted in Algorithm~\ref{alg:GetShellSeq} is neither unique, nor arbitrary: vertices are deleted in increasing order of their shell indices, but not all vertices with the same shell index are interchangeable. For example, consider the graph in Figure~\ref{fig:SameDegree-a}, for which every vertex has shell index equal to 1: the first vertex deleted will be, by necessity, one of the vertices of degree~1, but the second vertex deleted can vary depending on the choice of the first vertex.

Our sampling algorithm will generate graphs with vertices in an order that is compatible with Algorithm~\ref{alg:GetShellSeq}, so we will need to know more about such orderings. To that end, we give a simple condition for a graph $g$ on vertices  $\{v_1,\ldots,v_n\}$ that determines whether Algorithm~\ref{alg:GetShellSeq} could potentially process its vertices in that order, yielding a pre-specified sorted shell sequence $s_1\le\ldots\le s_n$.
\begin{condition}\label{c:sorted sequence}\rm
For all $i\in[n]$:
\begin{enumerate}
\item $v_i$ has at least $s_i$ neighbors $v_j$ with $s_j\ge s_i$, and
\item $v_i$ has at most $s_i$ neighbors $v_j$ with $j>i$.
\end{enumerate}
\end{condition}

\begin{lm}\label{L:sorted sequence}
Consider any graph $g\in G_n$ on vertices labeled $v_1,\ldots,v_n$ and
sorted sequence of $n$ non-negative integers $s_1\le\ldots\le s_n$.
Algorithm~\ref{alg:GetShellSeq} can process the vertices of $g$ in the given order,
yielding shell indices $s(v_i)=s_i$ for all $i\in[n]$,
if and only if $g$ satisfies Condition~\ref{c:sorted sequence}.
\end{lm}

\begin{proof}
Consider Algorithm~\ref{alg:GetShellSeq} on a graph $g$ satisfying Condition~\ref{c:sorted sequence},
at the moment when $s^*$ increments from $s-1$ to $s$.  The subgraph induced by $\{v_i:s_i\ge s\}$ has
minimum degree at least $s$ by Condition~\ref{c:sorted sequence}(i), so none of those vertices can have
been deleted yet.  On the other hand, if $v_i$ is the vertex remaining in $g$ with smallest index $i$, then
$v_i$ must have at least $s$ neighbors $v_j$ with $j>i$, so by Condition~\ref{c:sorted sequence}(ii),
$s_i\ge s$.  Thus, the vertices remaining in $g$ at that moment are precisely those $v_i$ with $s_i\ge s$.
Applying the argument for any $s$ and for $s+1$ shows that the vertices $v_i$ with $s_i=s$ are precisely
those which Algorithm~\ref{alg:GetShellSeq} deletes when $s^*=s$, as required.

For the other direction, suppose that Algorithm~\ref{alg:GetShellSeq} processes the vertices of $g$ in order, yielding $s(v_i)=s_i$ for all $i\in[n]$.
Then Condition~\ref{c:sorted sequence}(ii) is true since $s^*=s_i$ when $v_i$ is deleted.
Suppose that Condition~\ref{c:sorted sequence}(i) is not true for some $v_i$.
Just before $s^*$ increments from $s_i-1$ to $s_i$, all vertices $v_j$ with $s_j<s_i$ have been deleted, so $v_i$
has fewer than $s_i$ neighbors remaining.  Then $v_i$ could be deleted, which would make its shell index  $s_i-1$
according to the algorithm, a contradiction.
\end{proof}

Given a sorted shell sequence $s_1,\ldots,s_n$ of some simple graph, we initially aim to construct a graph $g$ in $n$ steps, by adding edges from
$v_i$ to $v_j$ with $j>i$ during Step~$i$ so that Condition~\ref{c:sorted sequence} is satisfied.  At Step~$i$, we will need to know how many neighbors $v_i$ already has with shell index at least $s_i$---call this number $t_i$.  Then Condition~\ref{c:sorted sequence} can be restated
as follows:
$v_i$ has between $s_i-t_i$ and $s_i$ new neighbors added during Step~$i$, where
 $t_i=|\{v_j:v_jv_i\in g,\ j<i,\ s_j\ge s_i\}|$.
These considerations are summarized in Algorithm~\ref{a:proto-generator}.
{\scriptsize
\begin{algorithm}[h]
\LinesNumbered
\DontPrintSemicolon
\SetAlgoLined
\SetKwInOut{Input}{input}
\SetKwInOut{Output}{output}
\Input{a sequence of non-negative integers $s_1\le\ldots\le s_n$}
\Output{a graph $g$ on vertices $v_1,\ldots,v_n$ with shell sequence $s(g)=(s_1,\ldots,s_n)$}
\BlankLine
\For{$i \leftarrow 1$ \KwTo $n$}{
    Make $v_i$ adjacent to a set $S$ of vertices $v_j$ with $j>i$ such that $s_i-t_i\le |S| \le s_i$\;
    Update $t_j$ values as needed.
}
\caption{Graph sampler: initial version}
\label{a:proto-generator}
\end{algorithm}
}

However, Algorithm~\ref{a:proto-generator} could get stuck if it is unable to choose $S$ as required.
This problem will not happen as long as the number of vertices $v_j$ with $i<j\le n$ is at least $s_i-t_i$.
For Steps~$i\le n-s_n$, the number of such vertices $n-i$ satisfies $n-i \ge s_n \ge s_i \ge s_i-t_i$,
so the problem can only occur for $i>n-s_n$.  To avoid this,
we will modify those steps of the algorithm.

Consider $i\ge n-s_n$.  Since the number of vertices $v_j$ with $j>i$ is $n-i\le s_n$ and $s_i=s_n$,
the condition $s_i-t_i\le |S| \le s_i$ reduces to just $|S|\ge s_n-t_i$.  The number of vertices
in $\{v_j: j\ge n-s_n\}$ is $s_n+1$, including $v_i$, so $v_i$ has $s_n$ potential neighbors in that set.  Thus, for such $i$,
Condition~\ref{c:sorted sequence} is equivalent to Condition~\ref{c:sorted sequence alternative ending}, which is as follows:

\begin{condition}\label{c:sorted sequence alternative ending}\rm
For all $i\in[n]$ with $i\ge n-s_n$,
$v_i$ has at most $t_i$ non-neighbors in the set $\{v_j:n-s_n\le j\le n\}$.
\end{condition}

As we process vertices $v_i$ with $i\ge n-s_n$, let $t_i'$ represent
the maximum number of non-neighbors allowed among unprocessed vertices.
Initialize $t_i'=t_i$ for all $n-s_n\le i\le n$.
To satisfy Condition~\ref{c:sorted sequence alternative ending},
each $t_j'$ decreases by 1 whenever it is not made adjacent to the currently
active vertex $v_i$.
When a $t_j'$ reaches zero, we make it adjacent to all remaining vertices and
then remove $v_j$ from further consideration;
note that this does not change $t_i'$ for any $i\not=j$. 
Since no $t_i'$ will ever go below zero,
we will be able to process all $v_i$ with $i\ge n-s_n$ so that
Condition~\ref{c:sorted sequence alternative ending} is satisfied.

Finally, recall that $t_i=|\{v_j:v_jv_i\in g,\ j<i,\ s_j\ge s_i\}|$.
Since the given sequence $s_1,\ldots,s_n$ is sorted
in increasing order, $s_j>s_i$ is impossible when $j<i$.  Thus, an equivalent definition of $t_i$ is:
\begin{equation}
t_i = |\{v_j: v_jv_i\in g,\ j<i,\ s_j=s_i\}|.
\end{equation}

Algorithm~\ref{alg:ShellSeq} constructs graphs within the restrictions permitted by
Condition~\ref{c:sorted sequence} (for $i<n-s_n$) and Condition~\ref{c:sorted sequence alternative ending} (for $i\ge n-s_n$),
choosing randomly among all possibilities whenever there is more than one option.
As we have shown, the algorithm will never get stuck.
Thus, we have the following result:
\begin{thm}\label{thm:PosProbEveryG}
For any graph $g$ with shell sequence $s(g)$,  Algorithm~\ref{alg:ShellSeq} produces $g$, up to isomorphism, with positive probability.
\end{thm}

{\scriptsize
\begin{algorithm}[h]
\LinesNumbered
\DontPrintSemicolon
\SetAlgoLined
\SetKwInOut{Input}{input}
\SetKwInOut{Output}{output}
\Input{a sorted integer sequence $s_1\le\ldots\le s_n$}
\Output{a graph $g$ with shell sequence $s(g)=(s_1,\ldots,s_n)$}
\BlankLine
Initialize $v_1,\ldots,v_n$ to be the vertices of $g$.\;
Initialize $t_1=\ldots=t_n=0$\;
\For{$i\leftarrow 1$ \KwTo $n-s_n-1$}{
    Choose a random subset $R$ of $\{v_j:i<j\le n\}$ with $\max\{0,s_i-t_i\}\le |R| \le s_i$\;
    \For{$v_j\in R$}{
        Add the edge $v_iv_j$ to $g$\;
        \lIf{$s_j=s_i$}{$t_j\leftarrow t_j+1$}
    }
}
Initialize $S=\{v_j:n-s_n\le j\le n\}$\;
\For{$v_j\in S$}{
    \If{$t_j=0$}{
        $S\leftarrow S\setminus\{v_j\}$\;
        Add edges from $v_j$ to all $v_k\in S$ in $g$\;
    }
}
\While{$S\not=\emptyset$}{
    Pick any $v_i\in S$\;
    $S\leftarrow S\setminus\{v_i\}$\;
    Choose a random subset $R$ of $S$ with $|R| \ge |S|-t_i$\;
    \For{$v_j\in R$}{
        Add the edge $v_iv_j$ to $g$\;
    }
    \For{$v_j\in S\setminus R$}{
        $t_j\leftarrow t_j-1$\;
        \If{$t_j=0$}{
            $S\leftarrow S\setminus\{v_j\}$\;
            Add edges from $v_j$ to all $v_k\in S$ in $g$\;
        }
    }
}
\caption{Graph Sampler: construct a random graph with a given shell sequence}
\label{alg:ShellSeq}
\end{algorithm}
}

A comment on the running time of Algorithm~\ref{alg:ShellSeq}: Since a random set $R$ can be chosen from a given set $S$ in $O(|S|)$ time,  this algorithm
runs in $O(|V|^2)$ time.

We conclude this section by summarizing  simulation results.
 Algorithm~\ref{alg:ShellSeq}  randomly constructs both labeled graphs (which requires permuting the node labels of the output of the algorithm) and unlabeled graphs with a given shell distribution.
It produces graphs in every isomorphism class of the shell distribution,  and our simulations give preliminary evidence that it also does so quite fast.

As an example, consider
shell distribution $(0,2,1,4,0,0,0)$ on $7$ vertices. For labeled graphs, 10,000 runs of the algorithm produced more than 7,400 distinct graphs, which implies a very high discovery rate of the fiber. For unlabeled graphs, discovering the 12 isomorphism classes requires only 100 calls to the algorithm.

\section{Behavior of complementary statistics on the fiber of the shell ERGM}
 \label{sec:structural}

%I am still working on this and have more to add soon but this is what I have so far.  Please edit/makes suggestions as you see fit.  Dane 5/22

In this section, we explore, both theoretically and experimentally, the behavior of various subgraphs on the fiber of graphs with a given %core
shell distribution.  In the network literature, subgraphs---such as edges and triangles---are used to perform heuristic goodness-of-fit tests.  Hence, understanding how these subgraphs can vary across the set of graphs with a fixed shell %core
  distribution is important.  We present the results in terms of a \emph{sorted shell sequence}, but note that a sorted shell sequence is equivalent to a shell distribution, as one can be obtained from the other uniquely.  %We begin with results on the number of edges.
%\subsection{Edges}
%\sout{
%Let $s_1\le\ldots\le s_n$ be the shorted shell sequence of an observed graph $g$ where each $s_i$ is the shell index of vertex $v_i$.  Let $n_S(g)=(n_0,\ldots,n_{n-1})$ be the corresponding shell distribution, so that $n_j$ is the number of vertices $v_i$ with $s_i=j$.  Suppose degeneracy$(g)=m$, so that $s_n=m$ and $n_j=0$ for $m<j\le n-1$.  We are interested in determining how many edges a graph with the shell sequence $s_1\le\ldots\le s_n$ may have.
%}
%\textcolor{red}{Question: Did we really need all of that notation re-introduced here? wasn't it already used in the paper before and is, by now, familiar to the reader? - SP}
%% NEW TEXT:
The following are  lower and upper bounds on the number of edges and triangles in a graph with a prescribed shell sequence and degeneracy $m$.
\begin{proposition}\label{prop:maxedges}
If $g$ is a graph with sorted shell sequence $s_1\le\ldots\le s_n$, then the maximum number of edges in $g$ is $$\binom{m}{2}+\sum_{i=1}^{n-m}s_i.$$
\end{proposition}

\begin{proof} By Lemma~\ref{L:sorted sequence},
each vertex $v_i$ has at most $s_i$ neighbors $v_j$ with $j>i$,
and the total number of $v_j$ with $j>i$ is $n-i$.
We will construct a graph so that the first bound is realized for $v_i$
with $i\leq n-m$ and the second bound is realized for $i\geq n-m$;
thus, it has the maximum possible number of edges.

Begin with a complete graph $G_0$ on the $m$ highest indexed vertices, $v_{n-m+1},\ldots,v_n$.  Then for each $1\le i\le n-m$, add exactly $s_i$ edges from $v_i$ to $V(G_0)$.  This yields a graph with the desired number of edges.
\end{proof}
%Fix any sorted shell sequence $s_1\leq\ldots\leq s_n$, or equivalently, any shell distribution $(n_0,\ldots,n_k)$, where each $s_i$ is the shell index ofvertex $v_i$ in some graph, where $n_j$ is the number of vertices $v_i$ with
%shell index $s_i=j$, and $k$ is the degeneracy of that graph (which equals the maximum shell index $s_n$).

%We now consider the minimum number of edges in a graph with a prescribed shell sequence.

\begin{proposition}
If $g$ is a graph with sorted shell sequence $s_1\le\ldots\le s_n$ and corresponding shell distribution $n_S(g)=(n_0,\ldots,n_{n-1})$, then the minimum number of edges in $g$ is $$\sum_{j=1}^m f(n_j,j),$$ where
$$
f(n_j,j)=
\begin{cases}
\lceil\frac{jn_j}{2}\rceil & \text{if}\ j<n_j\\
jn_j-\binom{n_j}{2} & \text{if}\ j\ge n_j.
\end{cases}
$$
\label{prop:minedges}
\end{proposition}

\bigskip
%The construction for the maximum number of edges (for a graph with the given shell distribution) is the same as it is for the maximum number of triangles.  In terms of ``Michael's algorithm'', we get the maximum number of edges at every step. (I can prove that this greedy approach yields the global maximum.)
%
%Begin with a complete graph $G_0$ on the $k$ highest indexed vertices, $v_{n-k+1},\ldots,v_n$.  Then for each $1\leq i\leq n-k$, add exactly $s_i$edges from $v_i$ to $V(G_0)$.  The number of edges is:$$ {k\choose 2}+\sum_{i=1}^{n-k}s_i=\sum_{i=1}^{n}s_i -k^2 +{k\choose 2}=\sum_{i=1}^{n}s_i -{k+1\choose 2}=\sum_{j=1}^{k}jn_j -{k+1\choose 2}.$$(The last formula is written in terms of the shell distribution instead of the shell sequence, in case we someday decide we need it.)
\begin{proof}
For any $0\leq i\leq m$, the vertices with shell index $i$ must have at least $i$ neighbors in $\{v_j:s_j\geq i\}$.
We will construct a graph in stages as $j$ goes from $m$ down to $0$, adding vertices with shell index $j$
during stage $j$, using the minimum possible number of edges to satisfy the previous condition.

%We will construct a graph by adding vertices with shell index $j$ in several stages, from $j=m$ down to $j=0$, with the desired shell sequence in $m$ stages, so that the number of edges is minimizes after each stage.
%
First, given any $d<n$, we show how to construct a graph $G(n,d)$ with $n$ vertices, minimum degree $d$, and the fewest possible number of edges.    Let the vertex set be $Z_n$ and arrange the vertices evenly around a circle.  If $d$ is even, make each vertex adjacent to the $d/2$ closest vertices to it on either side.  If $d$ is odd and $n$ is even, make each vertex adjacent to the $(d-1)/2$ closest vertices to it on either side and also to the vertex directly across from it.  If $d$ is odd and $n$ is odd, there is no $d$-regular graph on $n$ vertices, but we can
construct an $n$-vertex graph with one vertex of degree $d+1$ and all other vertices of degree $d$, as follows:

Begin by making each vertex adjacent to the $(d-1)/2$ closest vertices to it on each side.  Then, for $0\le i\le {d-1 \over 2}$, make vertex $i$ adjacent to vertex $i+{d+1 \over 2}$.  Note that for $i={d-1 \over 2}$, we get an edge from vertex ${d-1 \over 2}$ to vertex $d \equiv 0 \bmod n$.  The degree of vertex~0 increases by two and every other vertex degree increases by one, as required.  Note that the number of edges in $G(n,d)$ is $\lceil{nd/2}\rceil$.

Now, start with $G(n_m,m)$, which we can do since $n_m\ge m+1$.  Next we consider $j$ starting from $j=m-1$ down to $j=0$, adding $n_j$ vertices with shell index $j$ at each step as follows:

If $n_j>j$, then we add a disjoint copy of $G(n_j,j)$.  If $n_j\le j$, we add a disjoint complete graph on $n_j$ vertices and, from each of its vertices, add edges to exactly $j-n_j+1$ other vertices (which were added at
earlier steps).

Let $f(n_j,j)$ be the number of edges added in Step $j$. Then $f(n_j,j)=\lceil{jn_j/2}\rceil$ when $j<n_j$ and
$$f(n_j,j)={n_j\choose2} + n_j(j-n_j+1)=jn_j-{n_j\choose2}$$
when $j\ge n_j$.
%If desired, one can check that $$f(n_j,j)=\max\left\{jn_j-{n_j\choose2},\lceil{jn_j/2}\rceil\right\}$$ which is cute but less useful.

The minimum number of edges is thus $\sum_{j=1}^k f(n_j,j)$.

\end{proof}

%\subsection{Triangles}
We now study the behavior of the number of triangles % for graphs with a fixed shell distribution.  We start
starting
with a sharp upper bound.

\begin{proposition} The maximum number of triangles for a graph with sorted shell sequence $s_1\le\ldots\le s_n=m$ is $$\binom{m}{3}+\sum_{i=1}^{n-m}\binom{s_i}{2}.$$
\label{prop:maxtri}
\end{proposition}

\begin{proof}  The construction in the proof of Proposition~\ref{prop:maxedges}  %\sidecomment{missing proposition reference}
produces a graph with the right number of triangles
and the argument is similar.
\end{proof}

Obtaining an explicit lower bound for the number of triangles is difficult.  Instead, we construct graphs with the given shell sequence with relatively few triangles, thus providing an upper bound for the minimum number of triangles for graphs with the specified shell sequence.  The first construction begins with a complete graph on $m$ vertices but then minimizes additional edges added in subsequent steps.

\begin{lm} Let $s_1\le\ldots\le s_n$ be a sorted shell sequence.  Then, there exists a graph $g$ with this shell sequence and exactly $A$ triangles, where $$A = \binom{s_n}{3} + \sum_{i=\max(1,n-2s_n+1)}^{n-s_n}\binom{s_i}{2}.$$
\label{lm:lemmaA}
\end{lm}

\begin{proof}Start with a complete graph on vertices $S_0:=\{ v_i: n-s_n+1\leq i\leq n \} $.
Let $S_1:=\{ v_i: \max(1,n-2s_n+1) \leq i \leq n-s_n \}$ and for each $v_i\in S_1$, add exactly $s_i$ edges from $v_i$ to $S_0$.
Finally, for $1\leq i\leq n-2s_n$, add to the graph a vertex $v_i$ and exactly $s_i$ edges from $v_i$ to $S_1$.
\end{proof}

The idea in the next construction is to grow a (nearly balanced) bipartite graph with partite sets
%\sidecomment{sorry, what are `partite sets'? parts? -SP  {\it The vertex set of a bipartite graph can be partitioned into two independent sets, which are often called `partite sets'. -MJP}}
 $S,S'$ rapidly.
However, it may be impossible to make a bipartite graph, so we maintain another set $S_0$ for the vertices that cannot be placed into $S$ or $S'$.  Every triangle will have at least one vertex in $S_0$.

Fix any sorted shell sequence $s_1\leq\ldots\leq s_n=m$.
If $n_m\ge 2m$, let $S_0=\emptyset$, let $S,S'$ be sets of sizes $\lfloor{n_m/2}\rfloor, \lceil{n_m/2}\rceil$, and let $G$ be the complete bipartite graph with partite sets $S,S'$.

Otherwise, $m+1\leq n_m<2m$.  Let $a_0 = 2s_n - n_m$ and let $a_m = a_m' = n_m - s_n$, then
let $S_0, S, S'$ be vertex sets of sizes $a_0,a_m,a_m'$ respectively.
Initialize $G$ to be the union of a complete graph on $S_0$ and the complete tripartite graph with partite sets $S_0,S,S'$.
Note that $G$ has $n_m$ vertices and minimum degree $s_n$.

Starting with $j = m-1$ and decreasing $j$ after each step, add $n_j$ vertices to $S\cup S'$, split so that that $|S|-|S'|$ is 0 or $\pm 1$.
Make each new vertex in $S$ adjacent to $j$ vertices in $S'$ if $|S'|\ge j$.  Otherwise, make each new vertex in $S$ adjacent to every vertex in $S'$ and also adjacent to $j-|S'|$ vertices in $S_0$; this adds ${j - |S'| \choose 2}$ triangles per new vertex in $S$.  Similarly add $j$ edges from each new vertex of $S'$ to vertices in $S$ if possible or to vertices in $S\cup S_0$ otherwise, which adds ${j - |S| \choose 2}$ triangles per new vertex of $S'$.
Let $B$ be the number of triangles in the graph obtained.

Although we cannot give a simple formula for $B$, we can compute $B$ directly, without actually constructing the graph:
If $n_m\geq 2n$, then $B=0$.  Otherwise, $m+1 \leq n_m < 2m$.  In that case, we use $a_j,a_j'$ to represent the sizes of $S,S'$ after step $j$,
where $j$ is initialized to be $m$ and then decreases after each step.  The number of new vertices in $S,S'$ in each step is represented by $x,x'$.
Then the computation can be performed as follows.
%\textcolor{red}{\\to answer the question from the margin comment: It wasn't clear to me on the previous readings. But I think it is interesting and we should keep it. However, I'm not sure if we should write it in the form of an algorithm, though that would certainly make it look better.  I defer the decision to someone else. -SP}
%\sidecomment{The point of this is that we can compute the bound without constructing the graph, which saves a bit of time (and see comment after the algorithm) - Is that being made clear?  Is it interesting enough to include?  If we keep it, should we rewrite it in the form of an algorithm?  -MJP}

{\scriptsize
\begin{algorithm}
\DontPrintSemicolon
\lIf{$n_m\geq 2n$}{$B\leftarrow 0$\;}
\uElse{
   Initialize $j\leftarrow m$, $a_0 \leftarrow 2 s_n - n_m$, $a_m =a_m' \leftarrow  n_m - s_n$, and
   $B \leftarrow {a_0 \choose 3} + {a_0 \choose 2}(a_m+a_m') + a_0a_ma_m'$\;
   \While{$j > 1$}{
      Let $j \leftarrow j-1$\;
      \lIf{$n_j$ is even}{$x\leftarrow n_j/2$ and $x'\leftarrow n_j/2$\;}
      \lElseIf{$a_{j+1}>a_{j+1}'$}{$x\leftarrow\lfloor{n_j/2}\rfloor$ and $x'\leftarrow\lceil{n_j/2}\rceil$\;}
      \lElse{$x\leftarrow\lceil{n_j/2}\rceil$ and $x'\leftarrow\lfloor{n_j/2}\rfloor$\;}
      $a_j\leftarrow a_{j+1}+x$ and $a_j'\leftarrow a_{j+1}'+x'$\;
      $B \leftarrow B + x {j - a_j' \choose 2} + x' {j - a_j \choose 2}$\ \tcc*[l]{where ${k \choose 2}=0$ whenever $k<2$}
   }
}
\caption{Compute $B$}
\end{algorithm}
}
Moreover, if ever $\min(j-a_j',j-a_j)<2$, then $B$ will remain fixed thereafter,
since $j$ is decreasing and $a_j$ and $a_j'$ are increasing evenly.
Thus, the algorithm can be terminated early if $\min (j-a_j',j-a_j)<2$.

%In the previous line, ${k \choose 2}$ is 0 whenever $k<2$.  Note that $j$ is decreasing and $a_j$ and $a_j'$ are increasing evenly, so $B$ will eventually no longer be increasing at all, at which point we may exit the while loop early and terminate.
%
%\label{lm:lemmaB}

\begin{proposition} Let $s_1\le\ldots\le s_n$ be a sorted shell sequence.  Then, the minimum number of triangles in a graph with this shell sequence is at most $\min\{A,B\}$.
\end{proposition}

\begin{proof} This follows immediately from Lemma~\ref{lm:lemmaA} and the previous construction.
\end{proof}

In order to further understand the behavior of these subgraph counts on the fibers of the model, we simulated graphs using Algorithm~\ref{alg:ShellSeq} with the shell distribution corresponding to the Sampson network studied above.  Here, we summarize the results of those simulations.

Recall that the $4$-truncated shell distribution of the Sampson network is $(0,2,3,13)$.  The network has 35 edges and 14 triangles.  Simulating 50,000 graphs with this shell distribution using Algorithm~\ref{alg:ShellSeq} produced graphs with as many as 41 and as few as 27 edges.  Propositions~\ref{prop:maxedges} and~\ref{prop:minedges} show that the maximum and minimum number of edges for graphs with this shell distribution are 44 and 24, respectively.  The maximum number of triangles among the simulated graphs was 30, and the minimum was 0.  The upper bound for the number of triangles in a graph with this shell distribution, as given by Proposition~\ref{prop:maxtri}, is 34.  The value $A$ in Lemma~\ref{lm:lemmaA} is 0, which coincides exactly with the minimum number of triangles observed in the simulations.

%The network science co-authorship network has shell distribution $n_S(g)=(0,27,87,94,102,23,21,16,
%\\
%9,0,\ldots,0)$, 914 edges and 921 triangles.  Using this shell distribution, 50,000 calls to Algorithm~\ref{alg:ShellSeq} produced graphs with between 1078 and 1178 edges.  According to Propositions~\ref{prop:minedges} and~\ref{prop:maxedges}, the lower and upper bounds number the number of edges for a graph with this shell distribution 659 and 1281, respectively.  For triangles, the maximum among simulated graphs was 500 and the minimum was 258.  Proposition~\ref{prop:maxtri} shows that in this case as many as 1946 triangles are possible, and the construction following Lemma~\ref{lm:lemmaA} shows that there exists a graph with this shell distribution and only 35 triangles.

It is worth noting that, among the 50,000 simulated graphs with shell distribution corresponding to that of the Sampson network, no two were isomorphic.  In other words, 50,000 calls to Algorithm~\ref{alg:ShellSeq} produced 50,000 distinct graphs.
%This was also true in the case of the network science co-authorship graph, a fact which
This again suggests that Algorithm~\ref{alg:ShellSeq} discovers the fiber of graphs with a fixed shell structure at a high rate.

\section{Existence of MLE and the model polytope}
\label{sec:distributionMLE}

It is well known   from the theory of exponential families (e.g., classical text \cite{BR:86})
that the MLE of the natural parameters of the model exists if and only if the average sufficient statistic of the sample lies in the interior of the following convex polyhedron. For discrete exponential families, and ERGMs in particular, \cite{RFZ:mle09}
offer details on the relevance of this polyhedron to the problem of maximum likelihood estimation and study its properties from both theoretical and algorithmic point of view.
\begin{defn}\rm
 The \emph{model polytope} (or \emph{marginal polytope}) for the shell distribution ERGM \eqref{eq:shellERGM} with the sufficient  statistic vector $(n_0(g),\dots,n_{m-1}(g))$ is the convex hull of all possible vectors of minimal sufficient statistics:
 \[
 	\mP_{n,m}=\conv\{(n_0(g),\dots,n_{m-1}(g)) | g\in\G_{n,m}\}\subset \mathbb R^{m}.
 \]
\end{defn}
Of course, each value of $m$ gives rise to a different polytope, but each  turns out to be a subpolytope (in fact, a face, as explained below) of the one with
unrestricted
degeneracy 
$m\leq n-1$.
Thus we define it as a special case and study its geometry first.
For simplicity of notation, denote  the  minimal sufficient statistic vector of the unrestricted model (i.e., the truncated shell distribution) by  $n^*_S(g)=(n_0(g),\ldots,n_{n-2}(g))$.
\begin{defn}\rm
 The model polytope for the shell distribution ERGM  with
 unrestricted
 degeneracy 
 is
 \[
 	\mP_{n}:=
	 \conv\{n_S^*(g)|g\in\G_n\}\subset \mathbb R^{n-1}.
 \]
\end{defn}

Denote by $\bar{n}^*_S$ the \emph{average sufficient statistic}  of the sample $g_1,\ldots,g_N$; its $j^{th}$ entry is $\frac{1}{N}\sum_{j=1}^{N}n_j^*(g_i)$.

\begin{proposition}
\label{prop:mleExistence}
	For a sample of size $N=1$, $\bar{n}^*_S$ never lies in the interior of $\mP_n$; that is, the MLE never exists.
\end{proposition}
\begin{proof}
Determining whether $\bar n^*_S$ lies in the relative interior of $\mP_{n}$ or on its boundary  requires an explicit description of the polytope. We will show that $\mP_n$ is a dilate of a simplex.
To this end, let us consider the polytope of non-truncated shell distributions:
\begin{align*}
	P_n=\conv\{ &(n_0,\dots,n_{n-1}): \\ & (n_0,\dots,n_{n-1})=n_S(g) \text{ for some } g\in\G_n\}.
\end{align*}
We claim that $(n_0,\dots,n_{n-1})=n_S(g)$ for some $g\in\G_n$ if and only if $n_m\ge m+1$ and $\sum n_j=n$, where $m=$ dgen$(g)$.

That $n_m\geq m+1$ is a necessary condition is clear by definition. That it is sufficient, it suffices to construct a graph $g$ with this sequence. But this is straightforward: starting with $K_m$, add $n_m-m$ vertices and connect each of them with every vertex of $K_m$. This gives the $m$-shell. Then, to construct the $j$-shell for all other $j$, simply add as many vertices as are necessary in the shell, and connect each of them with $j$ edges to some subset of the original $K_m$.

Listing all integer points of this polytope, it is not difficult to see that it  is simply an n-dilate of the simplex,
$	    P_n        = \conv\{ ne_i \} = n\Delta_{n-1}\subset \mathbb R^n,
$
where $e_i$ is the $i$-th standard unit vector in $\mathbb R^n$.
Finally, to obtain the polytope $\mP_{n}$ with the truncated sequences, simply omit the last coordinate from $P_n$. The only effect this has on the polytope is that it interprets the simplex $\Delta_{n-1}$ as living in $\mathbb R^{n-1}$, instead of the way it is written above, as a polytope in $\mathbb R^n$ but embedded in the hyperplane $\sum_j n_j=n$.

	Finally, note that all realizable integer points (i.e., those corresponding to a shell distribution) lie on the boundary of this polytope, and not its relative interior, since any realizable integer point must have a 0 in some component, as is evident from the necessary and sufficient conditions for shell distribution realizability given above. Thus, the MLE never exists for a single observation $g$.
\end{proof}
\begin{rmk}\rm
In case of larger samples, the MLE may or may not exist.
The decision requires checking if the average sufficient statistic is on the boundary of $\mP_n$.
\end{rmk}

We have shown that the polytope for
unrestricted
 degeneracy model, $\mP_n$, is just a dilate of the simplex, and all of the realizable sufficient statistics lie on its boundary. But the simple structure of $\mP_n$ also implies that  $\overline{\mP_{n,m}}\subset\mP_n$ for each $m\leq n-1$, where $\overline{\mP_{n,m}}$ denotes the embedding of $\mP_{n,m}$ into $\mathbb R^{n-1}$. Indeed, any point $p\in\mP_{n,m}\subset \mathbb R^{m}$ corresponds to a point $\overline{p}\in\mathbb R^{n-1}$ which is clearly a realizable shell distribution vector. Thus $\overline{p}$ is  a point in the polytope $\mP_n$ that lies on the face cut out by the equations that set all coordinates   other than $m$-th
 to zero.

\begin{rmk}
	Setting the degeneracy parameter $m$ to be equal to the observed graph and using the corresponding ERGM \eqref{eq:shellERGM} with sample space $\mathcal G_{n,m}$ behaves better than using
unrestricted degeneracy
 $m\leq n-1$ in general. In particular, many of the points that lie on the boundary of $\mP_n$ lie on the relative interior of a face of some $\mP_{n,m}$, thus the MLE has a positive probability of existing. The asymptotics of this construction are of interest to the behavior of the MLE problem, but are beyond the scope of the present paper.
\end{rmk}

\section{Discussion}
Cores have been widely used to study and summarize networks. In this paper we study the core decomposition of a network with an eye towards statistical inference. We embed the core structure of a network as captured by its shell distribution in the exponential random graph framework. We examine the theoretical properties of the model and study the problem of inference in the model which boils down to three tasks--existence of the MLE, sampling from the model and sampling from the fiber. The existence of MLE question is answered by characterizing the model polytope. To enable maximum likelihood estimation, we introduce a new type of support restriction that avoids bad behavior of the model common to many other classes of ERGMs. We develop an MCMC algorithm to sample from the model and apply this algorithm to estimate the MLE and perform heuristic goodness-of-fit tests. We also study the fiber which is the space of all graphs given a fixed shell distribution and develop a sampling algorithm that can generate any graph with a predefined core structure with positive probability. Further, we describe the fiber in detail by computing bounds on subgraph counts induced by fixing the core structure of a network.

Our experiments and theoretical results indicate that the shell distribution model captures information beyond the degree distribution and, in particular, the triadic effects quite well. 
The model support is obtained by conditioning on the degeneracy of a graph. Conditioning is common in ERGMs, as it improves model properties and stability of estimation algorithms. The choice of degeneracy and thus the specific shell ERGM depends on the data and is meant to provide a way to improve not only the model's stability, but also its interpretability.

\smallskip

There are several interesting extensions of this work worth pursuing. Inference in the shell distribution ERGM gives rise to several important problems that deserve attention. Firstly, even though the shell distribution of a network can be computed in linear time, when embedded in an MCMC algorithm to compute change statistics, this process is very slow. 
In contrast, the change statistics of most ERGMs can be computed locally, without the need of recomputing the new sufficient statistic of the entire graph. 
A natural question to ask is if one can compute the change statistics of the shell distribution more efficiently. In particular,  the following is of critical interest: is there a way to use the local change in the network, such as adding or deleting edges, to re-compute the shell distribution? 

A related question is on the proposal distribution used in the MCMC algorithm. Since we restrict the support of the model to graphs with degeneracy equal to $m$, it would be useful to find proposal distributions that generate networks that are always in this set.  We considered one type of summary statistic of the core distribution, namely the shell distribution and studied the associated ERGM thoroughly. Other interesting ways to summarize the core structure can be used to develop ERGMs. As mentioned, ERGMs based on the core distribution go beyond the dyadic assumption that is inherent in the degree-based analysis. An interesting summary statistic to consider is the degree of a node in its core. 

In a different direction, for many datasets, including the Sampson dataset, the network in question is directed. 
Notions of core decomposition can be defined for such generalizations of graphs as well: for example, the $(k,l)$-core of~\cite{DCores-Giatsidis2013} for directed graphs. It is not difficult to extend our model and algorithms to this notion of core decomposition, and it would be interesting to see how that model would perform.

  Finally, the support restriction applied to the core ERGM may be useful in other contexts, but a natural question to ask is how does one select the degeneracy parameter $m$. 

\section*{Acknowledgements}
The authors would like to thank Stephen Fienberg and Alessandro Rinaldo for several useful discussions on this topic and the  anonymous reviewers for their careful reading of our paper and their suggestions for clarifications. Petrovi\'c, Stasi and Wilburne were  supported by AFOSR/DARPA grant FA9550-14-1-0141. Karwa gratefully acknowledges support by a grant from the Singapore National Research Foundation under the Interactive and Digital Media Programme Office to the Living Analytics Research Centre.

\bibliographystyle{plainnat}
\bibliography{kCoresERGM}

\section{Appendix A} 
\label{sec:OldModel}

This appendix deals with the case when graph degeneracy $m$ is not restricted to one value for all graphs under the model. In other words,  the unrestricted model gives positive   probability to networks of degeneracy \emph{less than or equal to} any fixed value of $m\leq n-1$.
For simplicity, we will refer to this as the \emph{unrestricted} model, motivated by  the sample space restrictions placed in defining the core distribution ERGM in Section \ref{sec:distributionModel}.
 We will see that the choice of any particular such $m\leq n-1$ does not affect the behavior of the model; instead, problems arise when  allowing degeneracy to vary within the graphs in the model.
Section~\ref{sec:appendixModel} introduces the unrestricted model, which is ill-behaved (cf.\ Remark~\ref{rmk:degen}).  Section~\ref{sec:appendixDegenerate} explains this behavior.

\subsection{The model with unrestricted degeneracy}
\label{sec:appendixModel}
For completeness, let us  re-derive the model,  from first principles, for the unrestricted case $m \leq n-1$, for which the sample space is the set of all graphs with $n$ nodes, $\G_n$.

Again, to take advantage of the theory of exponential families, we rewrite Equation~\ref{eqn:graphprob} in exponential family form by
  re-parameterizing \mbox{$P(G=g)$} in terms of normalized probabilities $\tilde{p}_j=\frac{p_j}{p_{n-1}}$.
(Our notation very closely follows \cite{KayvanAleDK}.)

  Observe that $p_{n-1} =\frac{1}{1+\sum_{j=1}^{n-2}\tilde{p}_j}$, and thus  $P(G=g)$  can be written as
\begin{equation*}
P(G=g) = \varphi(p)\prod_{j=0}^{n-1}(\tilde{p}_jp_{n-1})^{n_j(g)}= \varphi(p)p_{n-1}^{\sum_{j=0}^{n-1}n_j(G)}\prod_{j=0}^{n-1}\tilde{p}_j^{n_j(g)} = \varphi(p)p^{n}_{n-1}\prod_{j=0}^{n-1}\tilde{p}_j^{n_j(g)},
\end{equation*}
or, more compactly, using that $\tilde{p}_{n-1}=1$ and  renaming the constant $\varphi(p)$ to $\phi(\tilde p)$ to reflect the re-parametrization:
\begin{align}\label{eq:Prob(G)}
P(G=g)=\frac{\phi(\tilde p)}{\left(1+\sum_{j=1}^{n-2}\tilde{p}_j\right)^n}\prod_{j=0}^{n-2}\tilde{p}_j^{n_j(g)}.
\end{align}
 Next,  let $\theta_j=\log\tilde p_j$ and define the normalizing constant in terms of $\theta$ as
$\psi(\theta)=n\log(1+\sum_{j=0}^{n-2}\exp(\theta_j))-\log(\phi(\tilde p)).$
With this, we can write $P(G=g)$ in exponential family form:
\begin{align}
P(G=g)=\exp\left\{\sum_{j=0}^{n-2}n_j(g)\theta_j-\psi(\theta)\right\}.
\end{align}
For this version of the model, the minimal sufficient statistic is given by the truncated shell distribution $n^*_S(g)=(n_0(g),\ldots,n_{n-2}(g))$.  As before,  it is not difficult to see that the natural parameter space $\Theta$ for the model is $\Theta=\mathbb{R}^{n-1}$.
%\sidecomment{used to say $N$ but we fixed to $n-1$; correct, yes?}
%  Thus, the model defines an algebraic map $\Theta\rightarrow\Delta_n$.

To obtain the log-partition function $\psi(\theta)$ in closed form, for fixed $n$,
consider the set of graphs on $n$ nodes as an ordered list, $\G_n=\{g_1=K_n,\ldots,g_i,\ldots,g_M=\bar{K}_n\}$, where the graphs are listed in non-increasing order in terms of the number of edges, and where $M=2^{\binom{n}{2}}$.
Note that in the empty graph $g_M$, every vertex has shell index $0$, while in the complete graph $g_1=K_n$,  the shell indices are $s(v)=n-1$ for all $v\in V(K_n)$. Therefore,
\begin{align}\label{eq:Prob(G=Kn)}
P(G=g_M)=\frac{\phi(\tilde p)}{\left(1+\sum_{j=0}^{n-2}\tilde{p}_j\right)^n}\cdot\tilde{p}_0^n,
\end{align}
and
\begin{align}
P(G=g_1) = \frac{\phi(\tilde p)}{\left(1+\sum_{j=0}^{n-2}\tilde{p}_j\right)^n}.
\end{align}
For any other arbitrary graph $g_i\in\G_n\setminus \{\bar{K}_n,K_n\}$,
\begin{align}\label{eq:Prob(G=gi)}
P(G=g_i)=\frac{\phi(\tilde p)}{\left(1+\sum_{j=1}^{n-2}\tilde{p}_j\right)^n}\prod_{j=0}^{n-2}\tilde{p}_j^{n_j(g_i)}.
\end{align}
Using $\sum_{i=1}^MP(G=g_i)=1$ and Equations~\eqref{eq:Prob(G=Kn)} and \eqref{eq:Prob(G=gi)}, the normalizing constant $\phi(\tilde p)$ can be rewritten as:
\begin{align}
\phi(\tilde p)=\frac{\left(1+\sum_{j=0}^{n-2}\tilde{p}_j\right)^n}{1+\ldots+\prod_{j=0}^{n-2}\tilde{p}_j^{n_j(g_i)}+\ldots+\tilde p_0^n}.
\end{align}
Finally, $\theta_j=\log\tilde p_j$ and the second equality in \eqref{eq:Prob(G=Kn)} provide
$
\psi(\theta)=\log(1+\ldots+\prod_{j=0}^{n-2}\tilde{p}_j^{n_j(g_i)}+\ldots+\tilde p_0^n) =
 \log(1+\ldots+e^{\sum_{j=0}^{n-2}n_j(g_i)\theta_j}+\ldots+e^{n\theta_0})
$.

\begin{example}\rm
Determining $\psi(\theta)$ for the case $n=3$ depends on counting simple graphs on three nodes up to isomorphism. Namely, there are $4$ non-isomorphic simple graphs on $3$ vertices (see Figure~\ref{fig:n=3case}): $\G_n$ consists of 1 copy of $g_1$, 3 isomorphic copies of $g_2$, 3 isomorphic copies of $g_3$ and 1 copy of $g_4$.  For $g_1=K_3$, each vertex has shell index 2, so $n_S^*(g_1)=(0,0)$.  For $g_2$, each vertex has shell index 1 and therefore $n_S^*(g_2)=(0,3)$.  Two vertices of $g_3$ have shell index 1 while the remaining vertex has shell index 0, so $n_S^*(g_3)=(1,2)$,  and $n_S^*(g_4)=(3,0)$ as every vertex of $g_4=\bar{K}_3$ has shell index 0.
Therefore,
the log-partition function for $n=3$ is
$
	\psi(\theta)=\log(1+3\tilde{p}_1^3+3\tilde{p}_0\tilde{p}_1^2+\tilde{p}_0^3)=  \log(1+3e^{3\theta_1}+3e^{2\theta_1+\theta_0}+e^{3\theta_0}).
$

\begin{figure}
\begin{subfigure}{.22\textwidth}
\begin{center}
\begin{tikzpicture}
  [scale=.4,auto=center,every node/.style={circle,fill=black,inner sep=1.5pt}]
  \path[use as bounding box] (0,-1) rectangle (4,4);
  \node (n1) at (0,0)  {};
  \node (n2) at (4,0)  {};
  \node (n3) at (2,2)  {};

  \foreach \from/\to/\weight in {n1/n2/1,n2/n3/1,n1/n3/1}
    \draw (\from) --(\to);

\end{tikzpicture}
\\
$n_S^*(g_1)=(0,0)$
\end{center}
\end{subfigure}
\begin{subfigure}{.22\textwidth}
\begin{center}
\begin{tikzpicture}
  [scale=.4,auto=center,every node/.style={circle,fill=black,inner sep=1.5pt}]
  \path[use as bounding box] (0,-1) rectangle (4,4);
  \node (n1) at (0,0)  {};
  \node (n2) at (4,0)  {};
  \node (n3) at (2,2)  {};

  \foreach \from/\to/\weight in {n2/n3/1,n1/n3/1}
    \draw (\from) --(\to);

\end{tikzpicture}
\\
$n_S^*(g_2)=(0,3)$
\end{center}
\end{subfigure}
%\smallskip
\begin{subfigure}{.22\textwidth}
\begin{center}
\begin{tikzpicture}
  [scale=.4,auto=center,every node/.style={circle,fill=black,inner sep=1.5pt}]
  \path[use as bounding box] (0,-1) rectangle (4,4);
  \node (n1) at (0,0)  {};
  \node (n2) at (4,0)  {};
  \node (n3) at (2,2)  {};

  \foreach \from/\to/\weight in {n2/n3/1}
    \draw (\from) --(\to);

\end{tikzpicture}
\\
$n_S^*(g_3)=(1,2)$
\end{center}
\end{subfigure}
\begin{subfigure}{.22\textwidth}
\begin{center}
\begin{tikzpicture}
  [scale=.4,auto=center,every node/.style={circle,fill=black,inner sep=1.5pt}]
  \path[use as bounding box] (0,-1) rectangle (4,4);
  \node (n1) at (0,0)  {};
  \node (n2) at (4,0)  {};
  \node (n3) at (2,2)  {};

  \foreach \from/\to/\weight in {}
    \draw (\from) --(\to);

\end{tikzpicture}
\\
\footnotesize{$ n_S^*(g_4)=(3,0)$}
\end{center}
\end{subfigure}
\caption{Truncated shell distributions of all non-isomorphic simple graphs on $3$ vertices.} \label{fig:n=3case}
\end{figure}
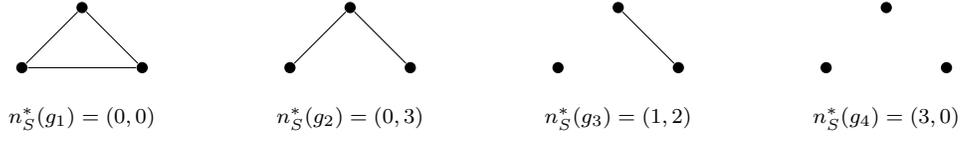
\end{example}

\subsection{Bad behavior of the unrestricted model}\label{sec:appendixDegenerate}

In this subsection we illustrate how the model misbehaves if the degeneracy $m$ is not controlled. 
The model with unrestricted degeneracy  allows,  for any fixed $m$,  the support of the model to contain graphs with degeneracy less than or equal to $m$, i.e. the sample space of the model is defined as follows: 
\[\G_{n,\leq m} =\{g\in\G_n: \dgen(g) \leq  m\}.\] 
Note that a special case is when $\G_{n,\leq n-1}=\G_n$, that is, a graph with \emph{any} degeneracy is allowed with positive probability under the model.

If we allow the model to put positive mass on graphs with degeneracy less than or equal to $m$, then for any generic point in the parameter space $\Theta$, the following behavior occurs. The likelihood function has many modes, and the local modes of the model corresponding to graphs where all nodes lie in the shells that are most popular (with respect to the $m^{th}$ shell). The example below illustrates this point, followed by Lemma~\ref{lemma:bad} that makes this intuitive explanation of the model  behavior precise.

%\paragraph{Example} 
\begin{example}\rm\label{ex:unrestrictedERGMbad}
Let $m = 4$ and consider the unrestricted shell ERGM supported on the sample space $\G_{n,\leq 4}$, i.e. the model puts a positive mass on all graphs with degeneracy less than or equal to $4$. Let $\theta = (\theta_0,\dots,\theta_4)$ 
%\}_{i=0}^4$ 
 be a parameter vector of this model. Recall that $\theta_i = \log  \frac{p_i}{p_m}$ and hence $\theta_4 = 0$. Without significant loss of generality, let us assume that $\theta_3 >\theta_0, \theta_1,\theta_2$. 
% REMOVED ()  !! 
 Hence amongst shells $0$, $1$, $2$ and $3$, the $3^{rd}$ shell has the highest  attractiveness,
  relative to the $4^{th}$ shell. Consider the set of  graphs  whose degeneracy is less than $m=4$, i.e. $\G_{n,\leq 3}$. Let $g$ be any graph in $\G_{n,\leq 3}$, then $n_s(g) = (n_0(g),n_1(g),n_2(g),n_3(g),0)$. Let $g^*$ be  any graph in $\G_{n,\leq 3}$, where all nodes lie in the shell $3$, which is the most attractive shell, i.e., $n_s(g^*) = (0,0,0,n,0)$. 
  
  Then $P(g^*) > P(g)$.  
%\begin{proof} 
Indeed, the following inequalities are straightforward: 
	\begin{align*}
	\log \frac{P(g^*)}{P(g)} &= \sum_{i=0}^{m-1}{\theta_i(n_i(g^*) - n_i(g)}) \\
	&= -\sum_{i=0}^{m-2}{\theta_i n_i(g)} + \theta_{m-1}\left(n-n_{m-1}(g)\right) \\
	&= -\sum_{i=0}^{m-2}{\theta_i n_i(g)} + \theta_{m-1}\left(\sum_{i=0}^{m-2}{n_i(g)}\right) \\
	&= \sum_{i=0}^{m-2}{n_i(g)(\theta_{m-1} - \theta_i)} > 0.
	\end{align*}
%\end{proof}
This should be interpreted as follows: Among the set of all graphs with degeneracy less than or equal to $3$, the most likely graph will be such that all nodes are in the shell index corresponding to the largest $\theta$. Thus, in some sense, the local mode is a ``degenerate'' mode (no pun intended!). 
  \end{example}
 
 In the above example, we could have chosen any $\theta_k$,  $k\neq m$, to be the most attractive shell, and the shell distribution of $g^*$ should be modified accordingly, i.e. $n_k(g^*) = n$ and $n_i(g^*) = 0$ for all $i \neq k$. Moreover, we could have considered the mode over any restricted sample space, not just $\G_{n,\leq 3}$. Lemma \ref{lemma:bad} illustrates this point by generalizing the example in several directions, in particular, by allowing there to be more than one `popular' shell. Let $m$ be the degeneracy of the model, let $\theta = (\theta_0,\dots,\theta_{m-1})$ be the parameter vector of the shell ERGM. Define $[m] = \{0,1, \ldots, m-1\}$. 
\begin{lm}
	\label{lemma:bad}
	Consider the shell ERGM on the sample space $\mathcal{G}_{n, \leq m}$ with parameter vector $(\theta_0,\dots,\theta_{m})$,
	%$\{\theta_i \}_{i=0}^{m}$,
	 where $\theta_m=0$ by definition. Let  $g$ be any graph in $\mathcal{G}_{n,\leq d}$ with degeneracy $d < m$, i.e., $n_i(g) = 0$ for all $i > d$. Let $\mathcal{L}_d = \{l \in [d] : \theta_l = \max_{i \in [d]} \theta_i\}$. Let $\mathcal{L}_d^c = [d] \backslash \mathcal{L}_d$.  Let $g^*$ be any network with degeneracy $d$ such that nodes exist only in the most popular shells, i.e. $n_i(g^*) =0$ for all $i \notin \mathcal{L}_d$. 
	 
	 Then, $P(g^*) > P(g)$.
\end{lm}
\begin{proof}
	Let $\theta^* = \max_{i \in [d]} \theta_i$, and consider the following, as in Example~\ref{ex:unrestrictedERGMbad}: 
%\sidecomment{Still to do: replace all $\forall$ with `for all'.}
	\begin{align*}
	\log \frac{P(g^*)}{P(g)} &= \sum_{i\in [d]}{\theta_i(n_i(g^*) - n_i(g)}) \\
	&= \sum_{i \in \mathcal{L}_d^c}\theta_i(0-n_i(g)) + \sum_{i \in \mathcal{L}_d}{\theta_i\left(n_i(g^*) - n_i(g)\right)} \\
	&= - \sum_{i \in \mathcal{L}_d^c}\theta_i n_i(g) + \theta^* \sum_{i \in \mathcal{L}_d}{(n_i(g^*) - n_i(g))} \\
	&= - \sum_{i \in \mathcal{L}_d^c}\theta_i n_i(g) + \theta^* \left(n - \sum_{i \in \mathcal{L}_d}{n_i(g)}\right)  \\ 	
	&= - \sum_{i \in \mathcal{L}_d^c}\theta_i n_i(g) + \theta^* \left(\sum_{i \in \mathcal{L}_d^c}{n_i(g)}\right) 
	%\mbox{( Since }\sum_{i \in [d]}n_i(g_1) = n  )
	\\
	&= \sum_{i \in \mathcal{L}_d^c}{n_i(g)(\theta^* - \theta_i)} > 0.
	\end{align*}
	The fourth equality holds since $ n_i(g^*) = 0$ for all $i \in \mathcal{L}_d^c$. The fifth equality holds because $\sum_{i \in [d]}n_i(g) = n$. 
\end{proof}
As an additional example of behavior explained in  Lemma~\ref{lemma:bad}, let $m=5$, $d=3$ and let $\theta = (a,\alpha,b,\alpha, c,0)$ where $\alpha > a,b,c$.
%$\alpha > (a,b,c)$ makes no mathematical sense. 
By Lemma \ref{lemma:bad}, among all graphs with degeneracy at most $3$, graphs with shell distribution $(0,k,0,n-k,0,0)$ are the modes, where $n-k \geq 4$. Thus, among $d$-degenerate graphs, only graphs where all nodes lie in the most popular shells are modes. These graphs are vastly different from each other in terms of their topological properties (e.g. density, number of triangles), yet they occur as modes of the same parameter vector. 

The reason why such a behavior occurs is that allowing graphs with degeneracy less than $m$ introduces a linear constraint on the shell distributions of these graphs. Thus to eliminate such a behavior, we define the model so that any graph with degeneracy less than $m$ has $0$ probability. Two consequences of this fact are that when fitting the shell ERGM to an observed graph, (1) $m$ cannot be larger than the observed degeneracy, and (2) graphs with degeneracy less than the observed degeneracy have $0$ probability. 

To see why  (1) holds, let $g$ be an observed graph with shell distribution $n_s(g)$ and degeneracy $\hat{m}$. Consider fitting the shell ERGM to $g$ by allowing $m >  \hat{m}$. If the sample space is $\mathcal{G}_{n,m}$, the observed graph has $0$ probability under the model! On the other hand, if we let the sample space be $\mathcal{G}_{n,\leq m}$ and we have $\hat{m} < m$, the observed network lies in the set $\mathcal{G}_{n, \leq \hat{m}} \subsetneq \mathcal{G}_{n, \leq m}$. Lemma \ref{lemma:bad} can be applied to show that the model has an undesirable property.  Let $supp(n_S) = \{i \in [m]: n_i(g)\neq 0\}$. Let $\Theta_g$ be a subset of the parameter space such that indices of largest value of $\theta$ correspond to $supp(n_S)$, i.e.,

$$\Theta_g = \{ \theta \in \Theta: \forall s \in supp(n_S), \theta_s = \max_{i \in [m]} \theta_i \}$$
By Lemma \ref{lemma:bad}, any parameter in $\Theta_g$ will have the observed graph $g$ as one of its modes. Moreover, these models will have several other modes that have shell distributions quite different from the observed graph. 

The above discussion shows that if we allow $m > \hat{m}$, there exist a large subset of the parameter space where the model misbehaves. A natural question to ask is the converse -  does there exists a parameter vector for which the observed graph is the only mode? An easy algebraic calculation in the example below shows even a weaker requirement of having the model assign higher mass to graphs with shell distributions vastly different from the observed shell distribution is not possible.

%\paragraph{Example} 
\begin{example}\rm
Let the observed shell distribution be $n_S(g) = (0,k,0,n-k)$, with $n-k \geq 4$ and $k > 0$. Hence the observed degeneracy is $\hat{m}=3$. Consider the  shell ERGM with $m=3$ and sample space $\mathcal{G}_{n,\leq 4}$. Consider two graphs $g_1$ and $g_2$ with shell distributions $(0,0,0,n,0)$ and $(0,n,0,0)$. We will show that there does not exist any point in the parameter space such that $P(g)> P(g_1)$ and $P(g) > P(g_2)$ simultaneously. To this end, let $\theta = (\theta_0,\theta_1,\theta_2,0)$ be any point in the parameter space. Note that 
	$\log \frac{P(g)}{P(g_1)} =  (\theta_1 - \theta_3)k$
and 	$\log \frac{P(g)}{P(g_2)} =  (\theta_3 - \theta_1)(n-k)$. For both these terms to be positive at the same time, we need $\theta_1 > \theta_3$ and $\theta_3 > \theta_1$ which is impossible. Moreover if $\theta_1 = \theta_3$, then the model places equal probability on the observed graph $g$ and $g_1$ and $g_2$, which is undesirable.
\end{example}

\end{document}